\documentclass[10pt]{article}

\usepackage{geometry}
\usepackage{graphicx}
\usepackage{amsmath}
\usepackage{enumerate}
\usepackage{amssymb}
\usepackage{color}
\usepackage{hyperref}
\usepackage{epsfig}
\usepackage{tikz}
\usepackage{pgfplots}
\usepackage{caption}
\usepackage{float}
\usepackage{subcaption}

\allowdisplaybreaks

\newtheorem{theorem}{Theorem}

\newtheorem{prop}{Proposition}

\newtheorem{remark}{Remark}

\newtheorem{Assumption}{H.\!\!}

\newenvironment{proof}{\begin{trivlist}
   \item[\hskip\labelsep{\it Proof.}]}{$\hfill\Box$\end{trivlist}}

\newcommand{\notiz}[1]{}

\def\R{\mathbb R} 

\title{An adaptive Euler-Maruyama scheme for McKean-Vlasov SDEs with super-linear growth and application to the mean-field FitzHugh-Nagumo model}   
\author{Christoph Reisinger\footnote{Mathematical Institute, University of Oxford, Andrew Wiles Building, Woodstock Road, Oxford, OX2 6GG, UK, E-mail: christoph.reisinger@maths.ox.ac.uk} , Wolfgang Stockinger\footnote{Mathematical Institute, University of Oxford, Andrew Wiles Building, Woodstock Road, Oxford, OX2 6GG, UK, E-mail: wolfgang.stockinger@maths.ox.ac.uk. The author is supported by an Upper Austrian Government fund.}}
\date{}
\begin{document}
\maketitle 
\begin{abstract}
In this paper, we introduce adaptive Euler-Maruyama schemes for McKean-Vlasov stochastic differential equations (SDEs) assuming only a standard monotonicity condition on the drift and diffusion coefficients but no global Lipschitz continuity in the state variable for either, while global Lipschitz continuity is required for the measure component only. We prove moment stability of the discretised processes and a strong convergence rate of $1/2$.
Several numerical examples, centred around a mean-field model for FitzHugh-Nagumo neurons,
illustrate that the standard uniform scheme fails and that the adaptive approach shows in most cases superior performance to tamed approximation schemes. In addition, we introduce and analyse an adaptive Milstein scheme for a certain sub-class of McKean-Vlasov SDEs with linear measure-dependence of the drift.  
\end{abstract} 


\section{Introduction} 


A McKean-Vlasov equation (introduced in \cite{MK,MK2}) for a $d$-dimensional process $X=(X_t)_{t \in [0,T]}$, with a given $T >0$, is an SDE where the underlying coefficients depend on the current state $X_t$ and, additionally, on the law of $X_t$, i.e., they have the form
\begin{equation}
\label{eq:MC}
    \mathrm{d}X_t =b(t,X_t, \mathcal{L}_{X_t}) \, \mathrm{d}t + \sigma(t,X_t,\mathcal{L}_{X_t}) \, \mathrm{d}W_t, \quad X_0 = \xi \in L_0^2(\mathbb{R}^d),
\end{equation}
where $W = (W_t)_{t \in [0,T]} $ is an $m$-dimensional standard Brownian motion and $\mathcal{L}_{X_t}$ denotes the marginal law of the process $X$ at time $t$. Here, for $p \geq 2$, $L_{0}^{p}(\mathbb{R}^d)$ denotes the space of $\mathbb{R}^d$-valued, $\mathcal{F}_0$-measurable random variables $Y$ satisfying $\mathbb{E}[|Y|^p] < \infty$.


The existence and uniqueness theory for strong solutions of McKean-Vlasov SDEs with coefficients of linear growth and Lipschitz conditions with respect to the state and the measure is well-established; see, e.g., \cite{AS} for the classical setting. For further specific existence and uniqueness results on weak and strong solutions of McKean-Vlasov SDEs we refer to \cite{BBP, HSS, MVA} and references cited therein. Most relevant to this work, in the case of super-linear drift and diffusion, it is known from \cite{RST,stock1} that McKean-Vlasov SDEs admit a unique strong solution under a monotonicity condition and Lipschitz continuity of the coefficients in the measure component.

Our motivation to study McKean-Vlasov SDEs with non-globally Lipschitz drift (and diffusion) is that several important models in practice involve mean-field terms and do not exhibit the classical global Lipschitz conditions on the coefficients of the SDE. For example, some models in neuroscience, such as the Hodgkin-Huxley and FitzHugh-Nagumo models \cite{BFFT,BO2}, or mean-field equations describing the behaviour of a (large) network of interacting spiking neurons \cite{DINRT}, do not satisfy the classical assumptions (e.g., linear growth) on the drift coefficient.

The simulation of McKean-Vlasov SDEs 
of the form \eqref{eq:MC}
typically involves two steps: a particle approximation and a time-stepping scheme.
For a seminal work on numerical schemes for McKean-Vlasov SDEs we refer the reader to \cite{BO}.

In the first step, at each time $t$, the true measure $\mathcal{L}_{X_t}$ is approximated by the empirical measure 
\begin{equation}
\label{emp-meas}
\mu_t^{\boldsymbol{X}^{N}}(\mathrm{d}x) := \frac{1}{N}\sum_{j=1}^{N} \delta_{X_t^{j,N}}(\mathrm{d}x),
\end{equation}    
where $(\boldsymbol{X}^{N}_t)_{t \in [0,T]} = (X^{1,N}_t, \ldots, X^{N,N}_t)_{t \in [0,T]}^{\top}$ (so-called interacting particle system) is the solution to the $(\mathbb{R}^{d})^N$-dimensional SDE with components 
\begin{equation}\label{eq:PS}
\mathrm{d}X_t^{i,N} = b(t,X_t^{i,N},\mu_t^{\boldsymbol{X}^{N}})  \, \mathrm{d}t + \sigma(t,X_t^{i,N},\mu_t^{\boldsymbol{X}^{N}}) \, \mathrm{d}W_t^{i}, \quad X_{0}^{i,N} = \xi^i, 
\end{equation}
and $(W^{i},\xi^{i})$ are independent copies of $(W,\xi)$, for $i \in  \lbrace 1, \ldots, N \rbrace$. Furthermore, $\delta_{x}$ 
denotes the Dirac measure at $x$. 

Key to the convergence as $N\rightarrow \infty$ is the concept of propagation of chaos.
Consider 
$N$ (independent) processes $X^{i}$ driven by independent Brownian motions $W^{i}$
\begin{equation}
\label{eq:sys}
\mathrm{d}X_t^{i} = b(t,X_t^{i},\mathcal{L}_{X^{i}_t})  \, \mathrm{d}t + \sigma(t,X_t^{i},\mathcal{L}_{X^{i}_t}) \, \mathrm{d}W_t^{i}, \quad X^{i}_0 = \xi^i, 
\end{equation}
which are standard McKean-Vlasov SDEs (and by uniqueness $\mathcal{L}_{X^{i}_t} = \mathcal{L}_{X_t}$).
Then, pathwise propagation of chaos refers to the property 
\begin{equation*}
\lim_{N \to \infty} \max_{i \in \lbrace 1, \ldots, N \rbrace} \mathbb{E} \left[\sup_{t \in [0,T]} |X_t^{i,N} -X_t^{i}|^2 \right] = 0.
\end{equation*}
This type of result is classical under global Lipschitz conditions (see, e.\ g., \cite{RC}), and has been proven in \cite[Proposition 3.1]{RES} under super-linear growth conditions of the drift.
While the rates for the propagation of chaos shown in \cite{RES} suffer from the curse of dimensionality
(i.e., the rate depends on $d$ and diminishes as $d\rightarrow \infty$)
\cite[Lemma 5.1]{DLR} gives for $p \leq 4$ and any $i \in \lbrace 1, \ldots, N \rbrace$ the bound
\begin{equation*}
\mathbb{E} \left[\sup_{t \in [0,T]} |X_t^{i,N}-X_t^{i}|^p \right] \leq C N^{-p/2},
\end{equation*}
where the constant $C>0$ is dimension-independent. However, in \cite[Lemma 5.1]{DLR}, the drift coefficient of the underlying McKean-Vlasov SDE is assumed to satisfy strong regularity assumptions (differentiability, along with Lipschitz and boundedness conditions on the state and measure components) and the diffusion coefficient is assumed to be constant, i.e., the framework is much more restrictive than the setting of the present article (in particular, the super-linear growth case is not covered in \cite{DLR}). 


As McKean-Vlasov equations typically arise from a mean-field approximation to a high- but finite-dimensional particle system, 
a particle approximation to the McKean-Vlasov equation simply reverses this step.
To be more concrete, for instance, the FitzHugh-Nagumo model for $N$ interacting neurons from $P$ populations with different characteristics is a system of $N$ three-dimensional SDEs of the form
\begin{align}
\label{FHN-intro}
    \mathrm{d}X_t^{i,N}  = f_{\alpha}(t,X_t^{i,N}) \, \mathrm{d}t & + g_{\alpha}(t, X_t^{i,N}) \, \begin{bmatrix} 
           \mathrm{d}W_t^{i}   \\
           \mathrm{d}W_t^{i,y} 
         \end{bmatrix}  \nonumber \\
         &  + \sum_{\gamma=1}^P \frac{1}{N_\gamma} \sum_{j,p(j)=\gamma} \Big( b_{\alpha \gamma}(X_t^{i,N},X_t^{j,N}) \, \mathrm{d}t 
         +  \beta_{\alpha \gamma}(X_t^{i,N},X_t^{j,N}) \, \mathrm{d}W_t^{i,\gamma} \Big), 
\end{align}
where $N = \sum_{\gamma=1}^P N_{\gamma}$ and $f_{\alpha}$ satisfies a monotone growth condition but is not globally Lipschitz continuous (see Example 4 in Section \ref{Example:NeuroFHN} for a detailed description).
It is shown in \cite{BFFT,BO2} that the mean-field limit $N\rightarrow \infty$ is described by the three-dimensional McKean-Vlasov SDE
\begin{align} 
\label{MF-FHN-intro}
    \mathrm{d} \bar{X}_t^{\alpha} = f_{\alpha}(t,\bar{X}_t^{\alpha}) \, \mathrm{d}t + \sum_{\gamma=1}^{P} \mathbb{E}_{\bar{Z}} [b_{\alpha \gamma}(\bar{X}_t^{\alpha},\bar{Z}_t^{\gamma})] \,\mathrm{d}t 
    + g_{\alpha}(t,\bar{X}_t^{\alpha}) \, \mathrm{d}W_t^{\alpha} + \sum_{\gamma=1}^{P} \mathbb{E}_{\bar{Z}} [\beta_{\alpha \gamma}(\bar{X}_t^{\alpha},\bar{Z}_t^{\gamma})] \, \mathrm{d}W_{t}^{\alpha \gamma}, 
\end{align} 
for $\alpha \in \lbrace 1, \ldots, P \rbrace$, where $\bar{Z}$ is a vector-valued process independent of $\bar{X}$ with the same law. Conversely, \eqref{FHN-intro} can be interpreted as a particle approximation to \eqref{MF-FHN-intro}.

The focus of this paper is a
 time-stepping scheme for the particle system in \eqref{eq:PS} over some finite time horizon $[0,T]$.
 We have in mind systems which arise from particle approximations of McKean-Vlasov equations, but equally the scheme is of interest for interacting particle systems, potentially high-dimensional ones, in their own right.
Note that at each time-step, the computational cost of simulating the particle system is usually $\mathcal{O}(N^2)$, as interaction terms due to the dependence of the coefficients on the empirical measure have to be computed for each particle. Strategies to reduce this cost include the projected particle method of \cite{BS}. 

In the context of classical SDEs, it is well-documented that the explicit Euler-Maruyama scheme (see \cite{KP}), for which strong convergence results of order $1/2$ are known under Lipschitz-type conditions, is generally not appropriate if we work with drift terms of super-linear growth. Specifically, it was shown in \cite{HJK2} for a certain SDE that the uniform time-step Euler-Maruyama scheme is not stable, i.e., the moments of the discretised process explode as the mesh-size tends to zero, even though a unique solution of the original SDE with bounded moments exists.
A similar phenomenon is observed in the context of particle systems in \cite{RES}, namely
that during the simulation some of the particles strongly oscillate and eventually diverge (as also illustrated by Figure \ref{fig:ParticleDivergence} in this work
for the FitzHugh-Nagumo model).

To overcome this problem in the setting of classical SDEs, several stable time-discretisations have been introduced, including tamed explicit Euler-Maruyama and Milstein schemes \cite{HJK, SA, GW}, balanced schemes \cite{TZ, ZM}, an explicit adaptive Euler-Maruyama method \cite{FG}, a truncated Euler-Maruyama method \cite{XM2} and an implicit Euler-Maruyama scheme \cite{HMS}. For these methods, stability of the discretised process and strong convergence results have been proven.
In \cite{kelly}, an explicit and a semi-implicit adaptive scheme for classical SDEs with non-globally Lipschitz continuous coefficients is studied, assuming
a monotonicity property resulting from a Khasminskii-type condition \cite{khasminskii2011stochastic}, and 
additionally differentiability of the coefficients. A backstop scheme with a strict upper and lower bound on the time-step immediately gives attainability of the final time $T$,
and convergence arbitrarily close to order 1/2 is shown for a choice of the time-step which depends solely on the growth of the drift term. An adaptive time-stepping Milstein method for SDEs with one-sided Lipschitz drift was developed in \cite{kelly2}, which also relies on such a backstop scheme.   

In \cite{RES}, the tamed Euler-Maruyama scheme and an implicit scheme were introduced for McKean-Vlasov SDEs with super-linearly growing drift and a diffusion coefficient which is globally Lipschitz in the state variable. Both coefficients are assumed to be Lipschitz continuous in the measure. 
In this setting, besides pathwise propagation of chaos, the paper establishes strong convergence of the time-discretised particle system to the SDE (\ref{eq:PS}). The recent papers \cite{stock2,stock,CH} introduce and analyse tamed Milstein schemes for (delay) McKean-Vlasov SDEs with a one-sided Lipschitz condition on the drift and a globally Lipschitz continuous diffusion coefficient (in both, the state and measure component), while \cite{stock1} extends this analysis to the framework of McKean-Vlasov SDEs with common noise, where the diffusion coefficient is also allowed to grow super-linearly. 
We refer to \cite{stock1} for examples of equations exhibiting such super-linear diffusion.

The first contribution of the current article is to introduce implementable adaptive time-stepping methods for McKean-Vlasov SDEs with super-linear growth in the drift and diffusion, assuming only a monotonicity condition (see (H.\ref{Assum:B}(\ref{Assum:B2})) in Section \ref{sec:Supiff} below).
The adaptive scheme has to cope with the fact that different time meshes will generally be used for different particles and the mean-field term needs to be approximated efficiently on all such time meshes, a difficulty not encountered for standard SDEs. 

First, we will present an explicit adaptive Euler-Maruyama scheme for the setting where the diffusion coefficient is globally Lipschitz in both components. Here, we need only a mild assumption on the choice of the adaptive time-step (in the spirit of \cite{FG}) in order to achieve stability of the scheme. Then, we 
remove the Lipschitz assumption on the diffusion coefficient 
and impose only a certain Khasminskii-type monotonicity condition. However, we will need stronger assumptions on the time-step function, in order to also control the growth in the diffusion coefficient.
The original analysis is not applicable in this more general setting and a new time-step control with a suitable stopping time argument is needed, which ultimately gives attainability of the final time and stability of the scheme.
This extension is achieved without the need for a (semi-implicit) backstop scheme if the proposed adaptive time-step becomes too small
and equally forms a contribution to the existing literature on adaptive time-stepping schemes for classical (non-measure dependent) SDEs. 

Moreover, we prove moment stability and strong convergence of order $1/2$ in the time-step. Several numerical examples in Section \ref{SECNum} demonstrate that our adaptive scheme often outperforms the tamed Euler-Maruyama scheme introduced in \cite{RES}, i.e., in many cases it gives more accurate numerical approximations or even achieves superior strong convergence rates.

Additionally, we present an adaptive Milstein scheme for a sub-class of McKean-Vlasov SDEs 
where the drift depends on the state and 
linearly on its law but the diffusion coefficient is only state dependent, see equation (\ref{MCVlasov}) below for details. We consider this special class firstly because
many relevant examples arising in practical applications have this form (see, e.g., \cite{CGM} and the references cited therein), and secondly
to avoid introducing a notion of derivatives for functions depending on probability measures (e.g., the Lions derivative; see \cite{CD,stock,CH} for details). We do not foresee any fundamental difficulties in extending the analysis to a more general setting, and leave it to future work.

A new difficulty for the computation and the analysis of adaptive schemes for particle systems, as mentioned above, is the approximation of the empirical measure at all adaptive time points. If the number of particles is large, as for particle approximations of McKean-Vlasov equations, this may result in a high computational cost if the empirical measure is re-evaluated at the smallest time-steps dictated by any particle. To avoid this, we present a scheme where the empirical measure is kept constant over larger, fixed time-steps. This procedure avoids this extra cost but still allows us to obtain the same stability and convergence results as for the more expensive scheme. We prove stability and convergence for both these Euler-Maruyama schemes in the super-linear drift case. We only give the stability proofs for super-linear diffusion and the Milstein scheme for the first variant; while the stability results for the second variant follow along similar lines as the results presented in this article, the proofs of strong convergence are subject to future research.

The remainder of this article is organised as follows: In Section \ref{SECIntro}, we describe the particle method and our adaptive schemes, and give the main convergence result (concerned with a globally Lipschitz continuous diffusion); the proof thereof is deferred to Section \ref{SECPro}. 
Section \ref{Sec:AddSchemes} introduces an adaptive Euler-Maruyama scheme for a more general class where we allow both coefficients to grow super-linearly, and an adaptive Milstein scheme for a more specific class of McKean-Vlasov equations with linear measure-dependence; the proofs are again given in Section \ref{SECPro}.
The numerical results for several standard test cases taken from the literature, notably the FitzHugh-Nagumo model, are presented in Section \ref{SECNum}. 

 We end this section by introducing some basic notations and give background results needed throughout this article. \\ \\
\noindent
\textbf{Preliminaries} \\ \\
\noindent
Let $(\mathbb{R}^d,\left \langle \cdot,\cdot \right \rangle, | \cdot |)$ represent the $d$-dimensional
Euclidean space, and let 
$(\Omega,\mathcal{F},\mathbb{F}=(\mathcal{F}_t)_{t \in [0,T]},\mathbb{P})$, for a given $T>0$, be a complete filtered probability space supporting a $m$-dimensional Brownian motion $(W_t)_{t \in [0,T]}$ with respect to $\mathbb{F}$. In addition, we use $\mathcal{P}(\mathbb{R}^d)$ to denote the family of all probability
measures on $(\mathbb{R}^d,\mathcal{B}(\mathbb{R}^d))$, where $\mathcal{B}(\mathbb{R}^d)$ denotes the Borel $\sigma$-field over $\mathbb{R}^d$, and define the subset of probability measures with finite second moment by
\begin{equation*}
\mathcal{P}_2(\mathbb{R}^d):= \Big \{ \mu\in\mathcal
{P}(\mathbb{R}^d) \Big| \ \int_{\mathbb{R}^d} |x|^2 \mu(\mathrm{d} x)< \infty \Big \}.
\end{equation*}
For all linear operators (e.g., matrices $A \in \mathbb{R}^{d \times m}$) appearing in this article, we will use the standard Hilbert-Schmidt norm denoted by $\| \cdot \|$.  

As metric on the space $\mathcal{P}_2(\mathbb{R}^d)$, we use the Wasserstein distance. For $\mu,\nu \in\mathcal{P}_2(\mathbb{R}^d)$, the Wasserstein distance between $\mu$ and $\nu$ is defined as
\begin{equation*}
\mathcal{W}_2(\mu,\nu) := \inf_{\pi \in \mathcal
{C}(\mu,\nu)} \left( \int_{\mathbb{R}^d \times \mathbb{R}^d} |x-y|^2 \pi(\mathrm{d} x,\mathrm{d} y) \right)^{1/2},
\end{equation*}
where $\mathcal{C}(\mu,\nu)$ means all the couplings of $\mu$  and
$\nu$, i.e., $\pi\in\mathcal {C}(\mu,\nu)$ if and only if
$\pi(\cdot,\mathbb{R}^d)=\mu(\cdot)$ and $\pi(\mathbb{R}^d,\cdot)=\nu(\cdot)$. Further, for $p \geq 2$, $\mathcal{S}^p([0,T])$ refers to the space of $\mathbb{R}$-valued continuous, $\mathbb{F}$-adapted processes, defined on the interval $[0,T]$, with finite $p$-th moments, i.e., processes $(X_t)_{t \in [0,T]}$ satisfying $\mathbb{E} \left[ \sup_{t \in [0,T]} |X_t |^p \right] < \infty$. 


\section{Time-stepping scheme for super-linear drift coefficient}\label{SECIntro}

In Section \ref{sec:PF}, we give a precise set of assumptions for the McKean-Vlasov equation \eqref{eq:MC} which guarantee well-posedness and are needed in the stability and strong convergence analysis of the time-stepping scheme introduced later in this section. Here, we first focus on a diffusion coefficient which is globally Lipschitz continuous in the state and measure component, while in Section \ref{sec:Supiff} we will allow a diffusion coefficient growing super-linearly in the state variable. The non-uniform Euler-Maruyama scheme will be presented in Section \ref{Sec:TS} and the adaptive choice of time-steps and main convergence result in Section \ref{sec:globLipTime}.
\subsection{Problem formulation}\label{sec:PF}
We consider the McKean-Vlasov equation defined by (\ref{eq:MC}), repeated here for convenience,
\begin{equation}\label{eq:MCeq}
    \mathrm{d}X_t =b(t,X_t, \mathcal{L}_{X_t}) \, \mathrm{d}t + \sigma(t,X_t, \mathcal{L}_{X_t}) \, \mathrm{d}W_t, \quad X_0 = \xi \in L_0^p(\mathbb{R}^d),
\end{equation}
where $W$ is a $m$-dimensional Brownian motion and $p \geq 2$. 
We impose the following assumptions for the coefficients: 
\begin{Assumption}\label{Assum:Ax}
The mappings $b:[0,T] \times \mathbb{R}^{d} \times \mathcal{P}_2(\mathbb{R}^d) \to \mathbb{R}^{d}$ and $\sigma: [0,T] \times \mathbb{R}^{d} \times \mathcal{P}_2(\mathbb{R}^d) \to \mathbb{R}^{d \times m}$ are measurable and satisfy:
\begin{enumerate}[(1)]
    \item \label{Assum:Ax1} 
    (Lipschitz condition for diffusion coefficient): There exists a constant $L >0$ such that
    \begin{equation*}
       \| \sigma(t,x,\mu)  - \sigma(t,x',\mu') \| \leq L(|x-x'| + \mathcal{W}_2(\mu, \mu')), \quad \forall  t \in [0,T], \ x \in \mathbb{R}^d \text{ and } \mu, \mu' \in \mathcal{P}_2(\mathbb{R}^d).
    \end{equation*}      
    \item \label{Assum:Ax2} 
    (One-sided Lipschitz condition for drift): There exists a constant $L >0$ such that
    \begin{equation*}
     \left \langle x-x', b(t,x,\mu) - b(t,x',\mu) \right \rangle \leq L |x-x'|^2, \quad \forall  t \in [0,T], \ x, x' \in \mathbb{R}^d \text{ and } \mu \in \mathcal{P}_2(\mathbb{R}^d).
    \end{equation*}    
    \item \label{Assum:Ax3} 
    (Lipschitz measure dependence of drift): There exists a constant $L > 0$ such that
    \begin{equation*}
    | b(t,x,\mu) -b(t,x,\mu') | \leq L\mathcal{W}_2 (\mu, \mu'), \quad \forall  t \in [0,T], \ x \in \mathbb{R}^d \text{ and } \mu, \mu' \in \mathcal{P}_2(\mathbb{R}^d).
    \end{equation*}
    \item \label{Assum:Ax4} 
    (Polynomial growth condition): There exist constants $L > 0$, $q >0$, such that
    \begin{equation*}
      |b(t,x,\mu) -b(t,x',\mu) | \leq L(1+|x|^q +|x'|^q)|x-x'|, \quad \forall  t \in [0,T], \ x, x' \in \mathbb{R}^d \text{ and } \mu \in \mathcal{P}_2(\mathbb{R}^d).
    \end{equation*}
    \item \label{Assum:Ax5}
     ($1/2$-H\"{o}lder-continuity in time): There exists a constant $C >0$ such that 
    \begin{equation*}
     | b(t,x,\mu) -b(t',x,\mu) | +  \| \sigma(t,x,\mu) - \sigma(t',x,\mu) \| \leq C |t -t' |^{1/2}, \quad \forall  t,t' \in [0,T], \ x \in \mathbb{R}^d \text{ and } \mu \in \mathcal{P}_2(\mathbb{R}^d).
    \end{equation*}    
     \item  \label{Assum:Ax6} 
     There exists a constant $C >0$ such that $|b(0,0,\mu)| + \|\sigma(0,0,\mu)\| \leq C, \quad \forall  \mu \in \mathcal{P}_2(\mathbb{R}^d)$.      
\end{enumerate}
\end{Assumption}
\begin{remark}
Note that the 
above assumptions imply the so-called monotone growth condition, i.e., there exist non-negative constants $C_1$ and $C_2$ such that for all $t \in [0,T]$ and all $\mu \in \mathcal{P}_2(\mathbb{R}^d)$
\begin{equation}
\label{eqn:growth}
       \left \langle x,b(t, x, \mu) \right \rangle + \frac{1}{2}\|\sigma(t,x,\mu)\|^2 \leq C_1 + C_2 |x |^2.
\end{equation}
This condition, frequently employed in the literature (see, e.g., \cite{RES}), is needed to guarantee the moment boundedness of our numerical scheme (\ref{eq:TDPS}) presented below.
\end{remark}

Under Assumptions (H.\ref{Assum:Ax}(\ref{Assum:Ax1}))--(H.\ref{Assum:Ax}(\ref{Assum:Ax5})) and for $\xi \in L_0^p(\mathbb{R}^d)$, it is known from \cite{RST} that the SDE (\ref{eq:MCeq}) has a unique solution $X \in \mathcal{S}^p([0,T])$. In Section \ref{sec:Supiff}, (H.\ref{Assum:Ax}(\ref{Assum:Ax1})) and  (H.\ref{Assum:Ax}(\ref{Assum:Ax2})) will be replaced by the weaker Assumption (H.\ref{Assum:B}(\ref{Assum:B2})) which enables us to consider non-globally Lipschitz diffusion coefficients.

We will now briefly illustrate that the interacting particle system defined in (\ref{eq:PS}) is indeed well-posed. Set 
\begin{equation*}
B(t,\boldsymbol{x}^{N}):=(b(t,x_1,\mu^{\boldsymbol{x}^{N}}),\ldots,b(t,x_N,\mu^{\boldsymbol{x}^{N}})),~~~\Sigma(t,\boldsymbol{x}^{N}):=\mbox{diag}(\sigma(t,x_1,\mu^{\boldsymbol{x}^{N}}),\ldots,\sigma(t,x_N,\mu^{\boldsymbol{x}^{N}})),
\end{equation*}
where $\mu^{\boldsymbol{x}^{N}}(\mathrm{d}x):=\frac{1}{N}\sum_{j=1}^N \delta_{x_j}(\mathrm{d}x)$ for
$\boldsymbol{x}^{N}:=(x_1,x_2,\ldots,x_N) \in \mathbb{R}^{dN}$. Note that
\begin{equation*}
\mathcal{W}_2^{2}(\mu^{\boldsymbol{x}^{N}},\mu^{\boldsymbol{y}^{N}})\le \frac{1}{N} \sum_{i=1}^{N}|x_i-y_i|^2 =
\frac{1}{N} |\boldsymbol{x}^{N}-\boldsymbol{y}^{N} |^2,~~~\boldsymbol{x}^{N},\boldsymbol{y}^{N} \in \mathbb{R}^{dN},
\end{equation*}
which is a standard bound for the Wasserstein distance between two empirical measures, see, e.g., \cite{RC}.
Hence, from Assumptions (H.\ref{Assum:Ax}(\ref{Assum:Ax1}))--(H.\ref{Assum:Ax}(\ref{Assum:Ax3})), we readily deduce that there exists a
constant $L>0$ such that
\begin{equation*}
\left \langle \boldsymbol{x}^{N}-\boldsymbol{y}^{N}, B(t,\boldsymbol{x}^{N})-B(t,\boldsymbol{y}^{N}) \right \rangle \leq  L |\boldsymbol{x}^{N}-\boldsymbol{y}^{N}|^2,~~~\|\Sigma(t,\boldsymbol{x}^{N})-\Sigma(t,\boldsymbol{y}^{N})\|^2 \leq
L|\boldsymbol{x}^N-\boldsymbol{y}^N|^2,
\end{equation*}
for any $\boldsymbol{x}^{N},\boldsymbol{y}^{N} \in \mathbb{R}^{dN}$, and the claim follows from \cite[Theorem 3.1.1]{PR}. In addition, we have the stability estimate $\max_{i \in \lbrace 1, \ldots, N \rbrace} \mathbb{E}[\sup_{t \in [0,T]} |X_t^{i,N}|^p] \leq C_{p,T}(1 + \mathbb{E}[|\xi^{i}|^p])$, for some constant $C_{p,T}>0$, see, e.g., \cite[Proposition 5.1]{RES}.

\subsection{Adaptive Euler-Maruyama scheme for super-linear drift}\label{Sec:TS}


For some integer $M>0$, we define a uniform step-size $h=T/M$. We start by stating the standard Euler-Maruyama scheme with uniform time-steps for (\ref{eq:PS}), for $n \in \lbrace 0, \ldots, M-1 \rbrace$,
\begin{equation}\label{EulerMcKean}
    X_{t_{n +1}}^{i,N,M} =  X_{t_n}^{i,N,M} +  b(t_n, X_{t_n}^{i,N,M}, \mu_{t_n}^{\boldsymbol{X}^{N,M}}) \, h +\sigma(t_n, X_{t_n}^{i,N,M}, \mu_{t_n}^{\boldsymbol{X}^{N,M}}) \, \Delta W_{t_n}^i,
\end{equation}
where $t_n =nh$, $\Delta W_{t_n}^i = W_{t_{n+1}}^i - W_{t_n}^i$, $ X_{0}^{i,N,M} =\xi^i$, and
\begin{equation*}
\mu_{t_n}^{\boldsymbol{X}^{N,M}}(\mathrm{d}x):= \frac{1}{N} \sum_{j=1}^{N} \delta_{X_{t_n}^{j,N,M}}(\mathrm{d}x).
\end{equation*}

Extending (\ref{EulerMcKean}), we introduce now an adaptive Euler-Maruyama scheme for the particle system (\ref{eq:PS}). A difficulty arises if we compute for each particle adaptive time-steps based on $X_t^{i,N}$, as
a different time mesh 
may result for each particle. 
Then, at a time point associated with a specific particle, we need an approximation of the empirical measure 
for the computation of the mean-field term in order to evaluate the update for the next time-step.
However, we do not have the value of each particle 
available at that time point. Hence, an approximation to the mean-field term is not readily computable. 

\medskip
We now propose a first scheme that allows the computation of an empirical measure $\mu_{t}^{\hat{\boldsymbol{X}}^N}$ 
for all $t$.
To do so, we start at $t=0$ with $ \hat{X}_{0}^{i,N} =\xi^{i}$ for all $i \in \lbrace 1, \ldots, N \rbrace$. At step $n \ge 0$, we compute for each particle the size of the adaptive time-step $h_n^{i}:=h^{\delta}(\hat{X}_{n}^{i,N})$,
with the time-step function $h^{\delta}$ introduced precisely below. We then perform the Euler-Maruyama update as
\begin{equation*}
\text{\bf  Scheme 1:} \qquad\qquad 
    \hat{X}_{t_{n +1}}^{i,N} =  \hat{X}_{t_{n}}^{i,N} +  b(t_{n}, \hat{X}_{t_{n}}^{i,N}, \mu_{t_{n}}^{\hat{\boldsymbol{X}}^N}) \ h_n^{\min} +\sigma(t_{n}, \hat{X}_{t_{n}}^{i,N}, \mu_{t_{n}}^{\hat{\boldsymbol{X}}^N}) \ \Delta W_{t_n}^i,
\end{equation*}
with step-size $h_n^{\min}:=\min \lbrace h_n^{1}, \ldots, h_n^{N} \rbrace $,
$t_{n+1} = t_{n} + h^{\min}_n$ and $\Delta W_{t_n}^i = W_{t_{n+1}}^i - W_{t_n}^i$,
and with
\begin{align}\label{eq:empMeasure}
\mu_{t_{n}}^{\hat{\boldsymbol{X}}^N}(\mathrm{d}x) :=  \frac{1}{N} \sum_{j=1}^{N} \delta_{\hat{X}_{t_n}^{j,N}}(\mathrm{d}x).
\end{align}
We will later use the notations $n_t:= \max \lbrace n: t_n \leq t \rbrace$ and $\underline{t}:= \max \lbrace t_n \leq t \rbrace$. This allows us to introduce the piecewise constant interpolant process $\bar{X}_{t}^{i,N} := \hat{X}_{\underline{t}}^{i,N}$ and measure $ \mu_{t}^{\bar{\boldsymbol{X}}^N} := \mu_{\underline{t}}^{\hat{\boldsymbol{X}}^N} $, and also the continuous interpolant processes
\begin{align*}
&\hat{X}_{t}^{i,N} :=  \hat{X}_{\underline{t}}^{i,N} + b(\underline{t}, \hat{X}_{\underline{t}}^{i,N}, \mu_{\underline{t}}^{\hat{\boldsymbol{X}}^N})(t-\underline{t}) + \sigma(\underline{t}, \hat{X}_{\underline{t}}^{i,N}, \mu_{\underline{t}}^{\hat{\boldsymbol{X}}^N})(W_{t}^i -W_{\underline{t}}^i),
\end{align*}
with the associated interpolant empirical measure $\mu_{t}^{\hat{\boldsymbol{X}}^N}$,
so that $(\hat{X}_t^{i,N})_{t \in [0,T]}$ is the solution to the SDE
\begin{equation}\label{AdaptiveMcKean}
    \mathrm{d} \hat{X}_{t}^{i,N} = b(\underline{t}, \bar{X}_{t}^{i,N}, \mu_{t}^{\bar{\boldsymbol{X}}^N}) \,  \mathrm{d}t + \sigma(\underline{t}, \bar{X}_{t}^{i,N}, \mu_{t}^{\bar{\boldsymbol{X}}^N}) \, \mathrm{d}W_t^i.
\end{equation}
The stability of this scheme is analysed in Proposition \ref{prop1}, while for the strong convergence analysis, we refer to Section 5.2 of the arXiv version.

In the above discretisation, we compute for each particle $i$ the size of the time-step, determine the smallest one and simulate each particle with the same (small) step-size in order to approximate the measure. This is computationally expensive, as it requires us to take the minimal time-step taken over a large number of particles. As $N\rightarrow \infty$, at each $t$ and for all $x$ there will almost surely be some $i$ for which $\hat{X}_{t}^{i,N}>x$, thus requiring a step-size adapted to this extreme scenario.
Typically, however, at any given time only a very small proportion of the particles will require such a small step-size for stability of the overall system. Choosing the smallest such time-step for the evolution of all particles introduces a dominant, unnecessary cost, which will significantly increase the complexity order. However, for a moderate number of particles this scheme is still tractable.

\medskip
We will now define a more practical scheme for the case that $N$ is large.
For our final numerical procedure, hence, we keep the empirical measure constant in an interval of  length $\delta T$, 
where $\delta = 1/M$ for some integer $M > 0$. 
This avoids having to recompute the measure at each adaptive time point.
In the sequel, let $k_n$ be the integer for which $t_n \in [k_n \delta T, (k_n+1)\delta T)$. Note that $k_n$ might have the same value for different $n$. 
Then, our adaptive Euler-Maruyama method becomes:
\begin{equation*}
\text{\bf  Scheme 2:} \qquad\qquad 
    \tilde{X}_{t_{n+1}}^{i,N} =  \tilde{X}_{t_n}^{i,N} +  b(t_n, \tilde{X}_{t_n}^{i,N}, \mu_{k_n\delta T}^{\tilde{\boldsymbol{X}}^N}) \, h_n^{i} +\sigma(t_n, \tilde{X}_{t_n}^{i,N}, \mu_{k_n \delta T}^{\tilde{\boldsymbol{X}}^N}) \, \Delta W_{t_n}^i,
\end{equation*}
where $h_n^{i}:=\min \lbrace h^{\delta}(\tilde{X}_{t_n}^{i,N}), (k_n+1)\delta T-t_n \rbrace$, $t_{n+1} = t_n + h_n^{i}$, $\Delta W_{t_n}^i = W_{t_{n+1}}^i - W_{t_n}^i$ and $ \tilde{X}_{0}^{i,N} = \xi^i$. Similar to above, we have
\begin{equation*}
\mu_{k_n\delta T}^{\tilde{\boldsymbol{X}}^N}(\mathrm{d}x) := \frac{1}{N} \sum_{j=1}^{N} \delta_{\tilde{X}_{k_n \delta T}^{j,N}}(\mathrm{d}x).
\end{equation*}

We introduce the piecewise constant interpolant process $\tilde{\bar{X}}_{t}^{i,N} = \tilde{X}_{\underline{t}}^{i,N}$, and measure $\mu_{t}^{\tilde{\bar{\boldsymbol{X}}}^N} = \mu_{k_n \delta T}^{\tilde{\bar{\boldsymbol{X}}}^N} $, 
and also the continuous interpolant 
\begin{equation*}
    \tilde{X}_{t}^{i,N} =  \tilde{X}_{\underline{t}}^{i,N} + b(\underline{t}, \tilde{X}_{\underline{t}}^{i,N}, \mu_{k_n\delta T}^{\tilde{\boldsymbol{X}}^N})(t-\underline{t}) + \sigma(\underline{t}, \tilde{X}_{\underline{t}}^{i,N}, \mu_{k_n\delta T}^{\tilde{\boldsymbol{X}}^N})(W_t^i -W_{\underline{t}}^i),
\end{equation*}
so that $(\tilde{X}_{t}^{i,N})_{t \in [0,T]}$ is the solution to the SDE
\begin{equation}\label{eq:TDPS}
    \mathrm{d} \tilde{X}_{t}^{i,N} = b(\underline{t}, \tilde{\bar{X}}_{t}^{i,N}, \mu_{t}^{\tilde{\bar{\boldsymbol{X}}}^N}) \, \mathrm{d}t + \sigma(\underline{t}, \tilde{\bar{X}}_{t}^{i,N}, \mu_{t}^{\tilde{\bar{\boldsymbol{X}}}^N}) \, \mathrm{d}W_t^i.
\end{equation}

Given this framework, we will impose in Section \ref{sec:globLipTime} conditions on the size of the time-step which will allow us to prove attainability of $T$, stability and strong convergence of the scheme. The complexity of the scheme and the main results are also given there.

\subsection{Choice of time-step function and main result}\label{sec:globLipTime}

Compared to the classical Euler-Maruyama scheme with a uniform step-size, we use adaptive steps $h^{\delta}(\hat{X}_{t_n}^{i,N})$ depending on
the state $\hat{X}_{t_n}^{i,N}$ of the $i$-th particle at the current ($n$-th) time-step. For each $0 < \delta \leq 1$, the \textit{time-step function} $h^\delta: \mathbb{R}^{d} \to \mathbb{R}^{+}$ hence controls the size of the time-step, which is relevant especially for large arguments to prevent instabilities. Convergence is then considered as $\delta \to 0$.


We make the following assumptions (based on Section 2 in \cite{FG}) on the function $h^\delta$:
\begin{Assumption}\label{Assum:AxS}
There exists a continuous function $h: \mathbb{R}^{d} \to \mathbb{R}^{+}$ for which
\begin{enumerate}[(1)]
    \item \label{Assum:AxS3} 
    the time-step function $h^{\delta}$ satisfies
        \begin{equation*}
        \delta \min(T,h(x)) \leq h^{\delta}(x) \leq \min(\delta T, h(x)),
    \end{equation*}
    \end{enumerate}
where  the following hold for $h$:
\begin{enumerate}[(1)]
\setcounter{enumi}{1}
   \item \label{Assum:AxS1} 
   there exist constants $a,b,c >0$, such that 
    \begin{equation*}
        h(x) \geq (a |x|^c +b)^{-1}, \quad \forall  x \in \mathbb{R}^d,
    \end{equation*}
    \item \label{Assum:AxS2} 
    and there exist constants $L_c, L_d>0$, such that
    \begin{equation*}
        \left \langle x,b(t, x, \mu) \right \rangle +\frac{1}{2} h(x) |b(t, x, \mu) |^2  \leq L_c |x |^2 + L_d, \quad \forall  t \in [0,T], \ x \in \mathbb{R}^d \text{ and } \mu \in \mathcal{P}_2(\mathbb{R}^d). 
    \end{equation*} 
    For the adaptive Milstein scheme (see Section \ref{sec:globLipMil} for details), the factor $1/2$ is replaced by $3/2$.
\end{enumerate}
\end{Assumption}
\label{sec:globLipComp}

Here, Assumption (H.\ref{Assum:AxS}(\ref{Assum:AxS3})) ties the dependence of $h^\delta(x)$ on $x$, for each $\delta$, to that of a generic function $h$.
Note that the schemes are fully defined in terms of $h^\delta$, while the existence of $h$ is only needed to ascertain the desired theoretical properties of the scheme.
Specifically, 
Assumption
(H.\ref{Assum:AxS}(\ref{Assum:AxS1})) lower bounds the speed of decay of the function $h$ as $|x| \to \infty$ and is necessary to derive a bound for the expected number of time-steps;
Assumption (H.\ref{Assum:AxS}(\ref{Assum:AxS2})) in conjunction with (H.\ref{Assum:Ax}(\ref{Assum:Ax4})) is used to control the super-linear growth of the drift term.
A suitable choice of $c$ in (H.\ref{Assum:AxS}(\ref{Assum:AxS1})) can be deduced from $q$ in Assumption (H.\ref{Assum:Ax}(\ref{Assum:Ax4})), while $a$ depends on the constant factors
in (H.\ref{Assum:Ax}(\ref{Assum:Ax4})), and $b$ simply eliminates a singularity at $x=0$.

Conversely, one can use an appropriately chosen $h$ to define $h^\delta$ systematically.
Consider for simplicity a one-dimensional setting where the drift asymptotically behaves like $-x^{q+1}$ (as per (H.\ref{Assum:Ax}(\ref{Assum:Ax4}))):
choosing first $h(x) = \min(1,|x|^{-q})$
and then $h^\delta$ as either the lower or upper  bound in the double inequality in (H.\ref{Assum:AxS}(\ref{Assum:AxS3})) gives a viable time-step function $h^\delta$.
We give an explicit example for the construction of $h^\delta$, with a suitable choice of $h$ satisfying (H.\ref{Assum:AxS}), in Section \ref{Example:GL}.

Given now the randomness of the step-size at time $t$ through its dependence on $\tilde{X}_{t}^{i,N}$, the number of time-steps $M_T^i$ up to time $T$ is a random variable for each particle $i$. The total expected complexity is therefore characterised by $\sum_{i=1}^N \mathbb{E}[M_T^i]$, and can be estimated as follows.

\begin{prop}\label{PropTS}
\label{prop:exp-nr-steps}
Under 
(H.\ref{Assum:AxS}) and for $(\tilde{X}_t^{i,N})_{t \in [0,T]}$ with bounded moments (uniformly in $t$) up to order $c$,
there exists $C>0$ independent of $\delta$ and $N$ such that 
\begin{equation}
\label{exp-nr-steps}
\mathbb{E} \left[ \frac{1}{N} \sum_{i=1}^{N} M_T^i  \right] \leq C \delta^{-1}.
\end{equation}
\end{prop}
\begin{proof}
As the upper and lower bounds in (H.\ref{Assum:AxS}(\ref{Assum:AxS1}))--(H.\ref{Assum:AxS}(\ref{Assum:AxS2})) only depend on the state variable (and not on the measure), the analysis of \cite[Lemma 2]{FG} is readily adapted to our setting assuming a moment bound.
\footnote{We will show this in Proposition \ref{prop1}. In our case, in contrast to \cite{FG}, where the initial data is deterministic
we require enough regularity of the random initial value.}
We have 
\begin{align*}
M_T^{i} & \leq 1 + T\sup_{t \in [0,T]} \frac{1}{h^{\delta}(\tilde{X}_{t}^{i,N})} + T\delta^{-1} \\
& \leq 1 + T\delta^{-1} \left( \sup_{t \in [0,T]} \max \lbrace h^{-1}(\tilde{X}_{t}^{i,N}) , T^{-1} \rbrace + 1 \right) \\
& \leq  T\delta^{-1} \left(a \sup_{t \in [0,T]} | \tilde{X}_{t}^{i,N} |^c + b + (1+\delta)T^{-1} + 1 \right).
\end{align*}
Note that the term $T\delta^{-1}$ in the first inequality appears because $h_n^{i}$ may take a value which is smaller than the step-size computed by $h^{\delta}$, but this happens at most once in each subinterval of length $\delta T$. 

This estimate implies (\ref{exp-nr-steps}) 
because of the assumed moment bound on $(\tilde{X}_t^{i,N})_{t \in [0,T]}$.
\end{proof}


\noindent
We now give our main result on strong 
convergence of the adaptive Euler-Maruyama scheme, Scheme 2, and recall pathwise strong propagation of chaos (from  \cite{RES}). To do so, we require the following definition:
\begin{equation*}
\varphi(N)=
\begin{cases}
N^{-1/2}, &\text{ for }  d<4, \\
N^{-1/2}\log N, &\text{ for }  d=4, \\
N^{-2/d}, &\text{ for } d>4.
\end{cases}
\end{equation*}
\begin{theorem}
\label{thm:main}
Let the SDE (\ref{eq:MCeq}) satisfy Assumption (H.\ref{Assum:Ax}), $\xi \in L^p_0(\mathbb{R}^d)$, for a sufficiently large $p$, and assume that the time-step function $h^{\delta}$ satisfies (H.\ref{Assum:AxS1}). Then, there exist constants $C, \widetilde{C} >0$ such that 
\begin{equation}
\label{eqn:mainthm}
   \max_{i \in \lbrace 1, \ldots, N \rbrace}  \mathbb{E}\left[ \sup_{t \in [0,T]} | X^{i}_t - \tilde{X}_t^{i,N} |^2 \right] \leq C \left( \delta + \varphi(N)
   \right)
    \le \widetilde{C} \left(\mathbb{E} \left[ \frac{1}{N} \sum_{i=1}^{N} M_T^i  \right]\right)^{-1} + C \varphi(N),
\end{equation}
where $X^{i}$ and $\tilde{X}^{i,N}$ are defined by (\ref{eq:sys}) and (\ref{eq:TDPS}), respectively.
\end{theorem}

The proof of the first inequality in \eqref{eqn:mainthm} is found in Section \ref{SECPro}, specifically Subsection \ref{sec:main}, while the second inequality follows immediately from
Proposition \ref{prop:exp-nr-steps}.

Theorem \ref{thm:main}
shows a strong convergence of order $1/2$ in terms of the expected average number of time-steps over all particles.
Also, recall that the additional complexity in terms of number of particles is, in the worst case, of order $N^2$. This results from evaluating potentially different functions of the empirical measure (with $N$ terms) for all particles. In many examples, like in our tests in Sections \ref{Example:FG} and \ref{Example:GL} where the interaction term is simply an expectation and identical for all particles, the complexity is $\mathcal{O}(N)$.

\section{Extension to super-linear diffusion and higher order schemes}\label{Sec:AddSchemes}
Here, we will discuss additional adaptive time-stepping schemes. Specifically, in Section \ref{sec:Supiff}, we propose a stable Euler-Maruyama scheme allowing for a super-linearly growing diffusion coefficient,
in addition to a super-linear drift term, subject only to a mild monotonicity assumption. In Section \ref{sec:globLipMil}, we will present an adaptive Milstein scheme for a special class of McKean-Vlasov equations where the diffusion only depends on the state of the process and is assumed to be globally Lipschitz continuous.

\subsection{Super-linear diffusion coefficient}\label{sec:Supiff}
We will consider an adaptive Euler-Maruyama scheme, defined as in (\ref{eq:TDPS}) and (\ref{AdaptiveMcKean}), in the setting where both the drift and diffusion coefficient grow super-linearly in the state component. 
To guarantee well-posedness of (\ref{eq:MCeq}) in this more general framework and stability of the resulting adaptive Euler-Maruyama scheme, we will need to replace (H.\ref{Assum:Ax}(\ref{Assum:Ax1})) and (H.\ref{Assum:Ax}(\ref{Assum:Ax2})) and by the following new set of assumptions:
\begin{Assumption}\label{Assum:B}
\begin{enumerate}[(1)]
\item \label{Assum:B1} 
    (Lipschitz measure dependence of diffusion coefficient): There exists a constant $L >0$ such that
    \begin{equation*}
       \| \sigma(t,x,\mu)  - \sigma(t,x,\mu') \| \leq L \mathcal{W}_2(\mu, \mu'), \quad \forall  t \in [0,T], \ x \in \mathbb{R}^d \text{ and } \mu, \mu' \in \mathcal{P}_2(\mathbb{R}^d).
    \end{equation*}      
\item \label{Assum:B2}
(Monotonicity condition): Let $p \geq 2$ given. There exists a constant $L_1>0$ such that
  \begin{equation}\label{as:coer}
       \left \langle x - x' ,b(t, x, \mu) - b(t, x', \mu) \right \rangle + \frac{p+1}{2}\|\sigma(t,x,\mu) - \sigma(t,x',\mu) \|^2 \leq L_1 |x -x' |^2,
\end{equation}  
for all $t \in [0,T]$, $x, x' \in \mathbb{R}^d$ and $\mu \in \mathcal{P}_2(\mathbb{R}^d)$.
\end{enumerate}
\end{Assumption}
Note that equation (\ref{as:coer}) and (H.\ref{Assum:Ax}(\ref{Assum:Ax6})) imply
\begin{equation*}
       \left \langle x,b(t, x, \mu) \right \rangle + \frac{p-1}{2}\|\sigma(t,x,\mu)\|^2 \leq C_1 + C_2 |x |^2, \quad \forall t \in [0,T],  \ x \in \mathbb{R}^d \text{ and } \mu \in \mathcal{P}_2(\mathbb{R}^d), 
\end{equation*}
for some $C_1, C_2 > 0$.

Under Assumptions (H.\ref{Assum:B}), (H.\ref{Assum:Ax}(\ref{Assum:Ax3}))--(H.\ref{Assum:Ax}(\ref{Assum:Ax6})) and for $X_0 \in L_0^p(\mathbb{R}^d)$, it is known from \cite[Theorem 2.1]{stock1} that the SDE (\ref{eq:MCeq}) has a unique solution $X \in \mathcal{S}^k([0,T])$ for any $k < p$. A similar monotonicity condition to \eqref{Assum:B2} is imposed in the works \cite{SA2,KUMAR} on tamed schemes for non-measure dependent equations with super-linearly growing diffusion coefficient. We refer the reader to \cite[Chapter 5]{XM} for a general discussion of such a monotonicity condition in the framework of SDEs with locally Lipschitz continuous coefficients.

To make the adaptive scheme stable in this framework, we require a modification of (H.\ref{Assum:AxS}(\ref{Assum:AxS2})): 
\begin{Assumption}\label{Assum:BS2}
 For some constant $C>0$ and $q$ from  (H.\ref{Assum:Ax}(\ref{Assum:Ax4})) we have
    \begin{align*}
        h(x) \leq C (1 + | x|^{3q})^{-1}, \quad \forall x  \in \mathbb{R}^{d}.
    \end{align*} 
\end{Assumption}

\begin{remark}\label{remark:SupeDiff}
Due to  (H.\ref{Assum:Ax}(\ref{Assum:Ax4})), (H.\ref{Assum:Ax}(\ref{Assum:Ax6})) and (H.\ref{Assum:B}(\ref{Assum:B2})), there exist constants $C_1, C_2>0$ (depending on the constants appearing in these three assumptions) such that
\begin{align*}
& | b(t,x,\mu) | \leq C_1(1+ | x |^{q+1}), \quad  \forall  t \in [0,T], x \in \mathbb{R}^d \text{ and } \mu \in \mathcal{P}_2(\mathbb{R}^d), \\
& \| \sigma(t,x,\mu) \|^2 \leq C_2(1+| x|^{q+2}), \quad  \forall  t \in [0,T], \ x \in \mathbb{R}^d \text{ and } \mu \in \mathcal{P}_2(\mathbb{R}^d).
\end{align*}
These growth estimates along with (H.\ref{Assum:BS2}) imply that for some constant $C_3>0$, 
\begin{align*}
& | b(t,x,\mu) | \| \sigma(t,x,\mu) \| (h(x))^{1/2} \leq C_3 (1+ | x |^2), \quad  \forall  t \in [0,T], \ x \in \mathbb{R}^d \text{ and } \mu \in \mathcal{P}_2(\mathbb{R}^d),
\end{align*}
which will be employed in Section \ref{sec:supDiff} to prove stability of the Euler-Maruyama time-stepping scheme. 

The condition (H.\ref{Assum:BS2}) is stronger than (H.\ref{Assum:AxS}(\ref{Assum:AxS2})).
Assume (H.\ref{Assum:Ax}(\ref{Assum:Ax4})) holds for some $q$, then $|b(t,x,\mu)| = \mathcal{O}(|x|^{q+1})$, and it is sufficient to choose 
$h(x)=\mathcal{O}(|x|^{-q})$, as noted earlier. The stronger decay imposed by (H.\ref{Assum:BS2}) does not lead to a worse complexity order.
\end{remark}

Note that the structure of Assumption (H.\ref{Assum:BS2}) still allows us to impose (H.\ref{Assum:AxS}(\ref{Assum:AxS1})), which guarantees the boundedness of the expected number of time-steps.

We formulate the strong convergence result only for Scheme 1. In Section \ref{sec:supDiff}, we will prove stability of the scheme only, as the convergence analysis follows closely the arguments presented in \cite[Theorem 3.1]{stock1} once stability is established. 
\begin{theorem}\label{thm:Diffsup}
Let $p>0$, $\xi \in L^k_0(\mathbb{R}^d)$, for a sufficiently large $k$ (depending on $p$), and let Assumptions (H.\ref{Assum:B}) and (H.\ref{Assum:Ax}(\ref{Assum:Ax3}))--(H.\ref{Assum:Ax}(\ref{Assum:Ax6})) hold. If the time-step function $h^{\delta}$ satisfies Assumptions (H.\ref{Assum:BS2}) and (H.\ref{Assum:AxS}(\ref{Assum:AxS3})), then there exists a constant $C>0$ such that
\begin{equation*} 
\max_{i \in \lbrace 1, \ldots, N \rbrace} \mathbb{E} \left[ \sup_{t \in [0,T]} | \hat{X}_t^{i,N} - X_t^{i,N}|^p \right] \leq C \delta^{p/2}, 
\end{equation*} 
where ${X}^{i,N}$ and $\hat{X}^{i,N}$ are defined by \eqref{eq:PS} and (\ref{AdaptiveMcKean}), respectively.
\end{theorem} 

\subsection{Adaptive Milstein scheme}\label{sec:globLipMil} 
Here, we consider for ease of demonstration specifically McKean-Vlasov SDEs of the form (see \cite{stock,CH} for tamed schemes in a more general framework) 
\begin{equation}\label{MCVlasov}
    \mathrm{d}X_t = \int_{\mathbb{R}^{d}} b(X_t,y)\,  \mathcal{L}_{X_t}(\mathrm{d}y) \, \mathrm{d}t + \sigma(X_t) \, \mathrm{d}W_t, \quad X_0 = \xi \in L_0^p(\mathbb{R}^d),
\end{equation} 
for some $p \geq 2$ and $b:\R^d \times \R^d \to \R^d$.

Using the standard particle method to approximate the true measure, we arrive at 
\begin{equation}
\label{mil_ip} 
    \mathrm{d}X_t^{i,N} = \frac{1}{N}\sum_{j=1}^N b(X_t^{i,N},X_t^{j,N}) \, \mathrm{d}t + \sigma(X_t^{i}) \, \mathrm{d}W_t^{i}, \quad X_0^{i,N} = \xi^{i},
\end{equation}
where $\xi^{i}$ are again independent copies of $X_0$. In the sequel, we use $\sigma = (\sigma_1, \sigma_2, \ldots, \sigma_m)$, where $\sigma_i = (\sigma_{1,i}, \ldots, \sigma_{d,i})^{\top}$, for $i \in \lbrace 1, \ldots, m \rbrace$.
For $x=(x_1, \ldots, x_d)$, we denote
$    L^{j_1}: = \sum_{k=1}^d \sigma_{k,j_1} \frac{\partial}{\partial x_k}$.

We then make the following assumptions on $b$ and $\sigma$:
\begin{Assumption}\label{Assum:AxM}
\begin{enumerate}[(1)]
    \item \label{Assum:AxM1} (Lipschitz condition for diffusion coefficient):
    There exists a constant $L>0$ such that 
    \begin{equation*}
     \| \sigma(x) -\sigma(x') \| \leq L | x- x' |, \quad \forall x,x' \in \mathbb{R}^d. 
     \end{equation*}         
    \item \label{Assum:AxM2} 
    Every $\sigma_i$, $i \in \lbrace 1, \ldots, m \rbrace$, is continuously differentiable and there exists a constant $L>0$ such that
     \begin{equation*}
      | L^{j_1} \sigma_{j_2}(x) - L^{j_1} \sigma_{j_2}(x') |  \leq L | x - x' |, \quad j_1, j_2 \in \lbrace 1, \ldots, m \rbrace, \quad \forall x,x' \in \mathbb{R}^d. 
     \end{equation*}
    \item \label{Assum:AxM3}  (One-sided Lipschitz condition for drift):
    There exists a constant $L>0$ such that 
     \begin{equation*}
     \left \langle x-x' , b(x,y)-b(x',y)  \right \rangle \leq L | x -x' |^2, \quad \forall x,x',y \in \mathbb{R}^d. 
     \end{equation*}
    \item \label{Assum:AxM4} (Polynomial growth condition):
    There exist constants $L,q>0$ such that
    \begin{equation*}
    |  b(x,y) - b(x',y) | \leq L (1 + | x |^{q} + | x' |^q)| x-x' |, \quad \forall x,x',y \in \mathbb{R}^d.
    \end{equation*}
    \item \label{Assum:AxM5} (Lipschitz condition for drift kernel):
     There exists a constant $L>0$ such that
    \begin{equation*}
     |  b(x,y) - b(x,y') | \leq L |y-y' |, \quad  \forall x,y,y' \in \mathbb{R}^d.
    \end{equation*} 
    \item \label{Assum:AxM6}  
    There exists a constant $C>0$ such that $|b(0,y)| \leq C, \quad  \forall y \in \mathbb{R}^d$.
\end{enumerate}
\end{Assumption}
These assumptions will guarantee the existence of moments of the processes defined by (\ref{eq:MilAdap}) and (\ref{TameMil}) below.

To prove a strong convergence result, additional assumptions on $b$ are needed:
\begin{Assumption}\label{Assum:AxAM1}
     The functions $\mathbb{R}^d \ni x \mapsto b(x, y)$ and $\mathbb{R}^d \ni  y \mapsto b(x,y)$ are continuously differentiable 
     for all $y \in \mathbb{R}^d$ and $x \in \mathbb{R}^d$, respectively,
     and there exist constants $L_1, L_2, r>0$ such that 
    \begin{align*}
    & \| \partial_x b(x,y) -  \partial_x b(x',y) \| \leq L_1 (1+|x |^r + | x'|^r) | x-x' |, \\
    & \| \partial_y b(x,y) -  \partial_y b(x,y') \| \leq L_2 | y-y' |, \quad  \forall x, y, x', y' \in \mathbb{R}^d.
    \end{align*}   
\end{Assumption}

In the sequel, we further impose for simplicity a commutativity condition on the diffusion matrix
\begin{equation}
\label{eq_comm}
L^{j_1} \sigma_{k,j_2} = L^{j_2} \sigma_{k,j_1}, \quad j_1, j_2 \in \lbrace 1, \ldots, m \rbrace, \quad k \in \lbrace 1, \ldots, d \rbrace.
\end{equation}
This implies that the so-called \emph{L\'{e}vy area} (see, e.g., \cite{KP}), the double stochastic integral appearing in the construction of the Milstein scheme, is zero. Then, the adaptive Milstein scheme has the form
\begin{align}\label{eq:MilAdap}
    \hat{Y}_{t_{n+1}}^{i,N} &=  \hat{Y}_{t_{n}}^{i,N} + \frac{1}{N} \sum_{j=1}^N b(\hat{Y}_{t_n}^{i,N},\hat{Y}_{t_n}^{j,N})h_n^{\min} + \sigma(\hat{Y}_{t_{n}}^{i,N}) \Delta W_{t_n}^{i} \notag \\
    & \qquad+ \frac{1}{2} \sum_{j_1,j_2=1}^m L^{j_1} \sigma_{j_2}(\hat{Y}_{t_{n}}^{i,N}) (\Delta W_{t_n}^{i,j_1} \Delta W_{t_n}^{i,j_2}  -\delta_{j_1,j_2} h_n^{\min}),
\end{align}  
where $h_n^{\min}:=\min \lbrace h_n^{1}, \ldots, h_n^{N} \rbrace$ and $t_{n+1} = t_{n} + h^{\min}_n$. Analogously to the adaptive Euler-Maruyama scheme, one can also define the process $\tilde{X}^{i,N}$ with piecewise constant measure over time intervals of length $\delta T$,
as well as the continuous-time interpolants of $\hat{Y}^{i,N}$ and $\tilde{Y}^{i,N}$.

We give a result for the stability and strong convergence of the process $(\hat{Y}_t^{i,N})_{t \in [0,T]}$, which is the Milstein version of Scheme 1,
but do not pursue the analogue of Scheme 2 for simplicity.


\begin{theorem}\label{thm:Mil}
Let $p>0$ and $\xi \in L^k_0(\mathbb{R}^d)$, for a sufficiently large $k$ (depending on $p$). Let Assumptions (H.\ref{Assum:AxM}) and (H.\ref{Assum:AxAM1}) hold and the time-step function satisfies Assumptions (H.\ref{Assum:AxS}(\ref{Assum:AxS2}))--(H.\ref{Assum:AxS}(\ref{Assum:AxS3})). Then, there exists a constant $C>0$ such that for the solutions ${X}^{i,N}$ and $\hat{Y}^{i,N}$ of \eqref{mil_ip} and \eqref{eq:MilAdap},
\begin{equation*} 
\max_{i \in \lbrace 1, \ldots, N \rbrace} \mathbb{E} \left[ \sup_{t \in [0,T]} | \hat{Y}_t^{i,N} - X_t^{i,N} |^p \right] \leq C \delta^p. 
\end{equation*} 
\end{theorem}
Recall that for the adaptive Milstein scheme Assumption (H.\ref{Assum:AxS}(\ref{Assum:AxS2})) has a different constant than for the adaptive Euler-Maruyama scheme as pointed out there. We will only prove moment boundedness of this scheme in Section \ref{proofMil},
as the rest of the strong convergence analysis then follows by analogous steps to those in the proof of the main results in \cite{stock,CH}.

\section{Examples and numerical illustration}
\label{SECNum}

In this section, we present some numerical examples for the adaptive schemes and compare them with tamed Euler-Maruyama and Milstein schemes (\cite{RES} and \cite{CH,stock}, respectively).

The first three tests are modifications of tests from the literature to illustrate certain features of the method: in Sections 
\ref{Example:FG} and \ref{Example:GL} by adding a mean-field term to SDEs with non-Lipschitz drift;
in Section \ref{Example:K} by adding a non-Lipschitz drift to a McKean-Vlasov SDE. In Section \ref{Example:NeuroFHN}, we apply the method to the FitzHugh-Nagumo model from neuroscience. Then, in Section \ref{sec:Mil}, we give a numerical illustration for the performance of our proposed adaptive Milstein scheme. 

The first two examples are initially formulated as McKean-Vlasov equations and then approximated by a particle system in the canonical way, while the remaining examples are directly formulated as particle systems. \\ \\
\noindent
\textbf{Tamed Euler-Maruyama scheme}  \\ 
\noindent
The tamed Euler-Maruyama scheme (see (\cite{RES}) for the particle system (\ref{eq:PS}) with $h= T/M$, where $M$ denotes the number of (uniform) time-steps, reads 
\begin{equation*} 
    X_{t_{n+1}}^{i,N,M} = X_{t_n}^{i,N,M} + \frac{b(t_n, X^{i,N,M}_{t_n}, \mu_{t_n}^{\boldsymbol{X}^{N,M}})}{1+M^{-\alpha} |b(t_n, X^{i,N,M}_{t_n}, \mu_{t_n}^{\boldsymbol{X}^{N,M}})|} h + \sigma(t_n, X^{i,N,M}_{t_n}, \mu_{t_n}^{\boldsymbol{X}^{N,M}}) \Delta W^{i}_{t_n},  \quad X_{0}^{i,N,M} = \xi^{i}, 
\end{equation*} 
for $i \in \lbrace 1, \ldots, N \rbrace$, where the $N$ Brownian motions are independent, $\alpha =1/2$ (to achieve strong convergence of order $1/2$) or $\alpha =1$ (to enable strong order $1$ for a constant diffusion term) and 
\begin{equation*}
\mu_{t_n}^{\boldsymbol{X}^{N,M}}(\mathrm{d}x) = \frac{1}{N} \sum_{j=1}^{N} \delta_{X_{t_n}^{j,N,M}}(\mathrm{d}x).
\end{equation*}

\noindent
\textbf{Tamed Milstein scheme} \\
\noindent 
The tamed Milstein scheme for (\ref{MCVlasov}) is a special case of the schemes in \cite{CH,stock} and 
has the form 
\begin{align}\label{TameMil}
    Y_{t_{n+1}}^{i,N,M} &=  Y_{t_{n}}^{i,N,M} +\frac{1}{N}\tilde{b}(Y_{t_n}^{i,N,M},Y_{t_n}^{j,N,M})h \nonumber \\ 
    & \quad + \sigma(Y_{t_{n}}^{i,N,M}) \Delta W_{t_n}^{i} + \frac{1}{2} \sum_{j_1,j_2=1}^k L^{j_1} \sigma_{j_2}(Y_{t_{n}}^{i,N,M}) (\Delta W_{t_n}^{i,j_1} \Delta W_{t_n}^{i,j_2}  -\delta_{j_1,j_2} h), \quad Y_{0}^{i,N,M} = \xi^{i},
\end{align} 
where $\delta_{j_1,j_2} =1$, for $j_1 =j_2$ and $\delta_{j_1,j_2}=0$, for $j_1 \neq j_2$, and
\begin{equation*} 
\tilde{b}(Y_t^{i,N,M},Y_t^{j,N,M}):= \frac{\sum_{j=1}^N b(Y_{t_n}^{i,N,M},Y_{t_n}^{j,N,M}) }{1+ h \left | \frac{1}{N}\sum_{j=1}^N b(Y_{t_n}^{i,N,M},Y_{t_n}^{j,N,M}) \right |}. 
\end{equation*} 

To illustrate the strong convergence behaviour in $h$ ($\delta$, respectively), we compute the root mean-square error (RMSE) by comparing the numerical solution at level $l$ (of the time-discretisation), denoted by $\hat{X}_T^{i,N,M_l}$ (which generically stands for a numerical approximation obtained by the tamed or adaptive schemes with $M=M_l$), with the solution at level $l-1$, at the final time $T$.   To be precise, we compute
\begin{equation*}
\sqrt{\frac{1}{N} \sum_{i=1}^{N} \left(\hat{X}_T^{i,N,M_l} - \hat{X}_T^{i,N,M_{l-1}} \right)^2},
\end{equation*}
where $M_l = \lceil 2^lT\rceil$ and 
the two particle systems at each level are generated by the same Brownian motions. 
For our numerical experiments, we set $T=1$.

\subsection{Example 1 -- the Fang and Giles test with added mean-field term} 
\label{Example:FG}

Consider the one-dimensional SDE
\begin{equation*} 
    \mathrm{d}X_t = \left(- \frac{X_t}{1-|X_t|^2} + \mathbb{E}[X_t] \right) \, \mathrm{d}t + \, \mathrm{d}W_t, \quad X_0 =0,
\end{equation*} 
adapted from \cite{FG} by including the term $\mathbb{E}[X_t]$. The tamed Euler-Maruyama scheme reads 
\begin{equation*} 
      X_{t_{n+1}}^{i,N,M} = X_{t_n}^{i,N,M} + \frac{b(t_n, X^{i,N,M}_{t_n},\mu_{t_n}^{\boldsymbol{X}^{N,M}})}{1+M^{-\alpha} |b(t_n, X^{i,N,M}_{t_n}, \mu_{t_n}^{\boldsymbol{X}^{N,M}})|} h + \Delta W^{i}_{t_n}, 
\end{equation*} 
where 
\begin{equation*} 
    b(t_n, X^{i,N,M}_{t_n},\mu_{t_n}^{\boldsymbol{X}^{N,M}})= -\frac{X^{i,N,M}_{t_n}}{1 - |X^{i,N,M}_{t_n} |^2} + \frac{1}{N} \sum_{j=1}^{N} X^{j,N,M}_{t_n}. 
\end{equation*}
As the diffusion is constant, one would expect a strong convergence order of $1$ for the tamed scheme. Hence, one would set $\alpha=1$ to achieve this rate. However, as we will illustrate later the choice $\alpha=1$ seems not to give any rate of convergence, while the stronger taming with $\alpha=1/2$ at least recovers the expected rate of order $1/2$.

Though this example does not satisfy the assumptions listed above, it was numerically shown in \cite{FG} for the corresponding SDE without the mean-field term that the adaptive Euler-Maruyama method with $h^{\delta}(x) = \delta(1 - |x|^2)$ outperformed the tamed Euler-Maruyama scheme and an implicit scheme. 

If a newly computed approximation $\hat{X}^{i,N}_{t_{n+1}}$ exceeds a predefined radius $r_{\text{max}} = 1 - 10^{-10}$, this value is replaced by $r_{\text{max}}$ and $(r_{\text{max}}/ | \hat{X}^{i,N}_{t_{n+1}} |) \hat{X}^{i,N}_{t_{n+1}}$ in the one- and multi-dimensional setting, respectively.
\begin{figure}[!h]
	\centering
  \includegraphics[width=0.4\textwidth]{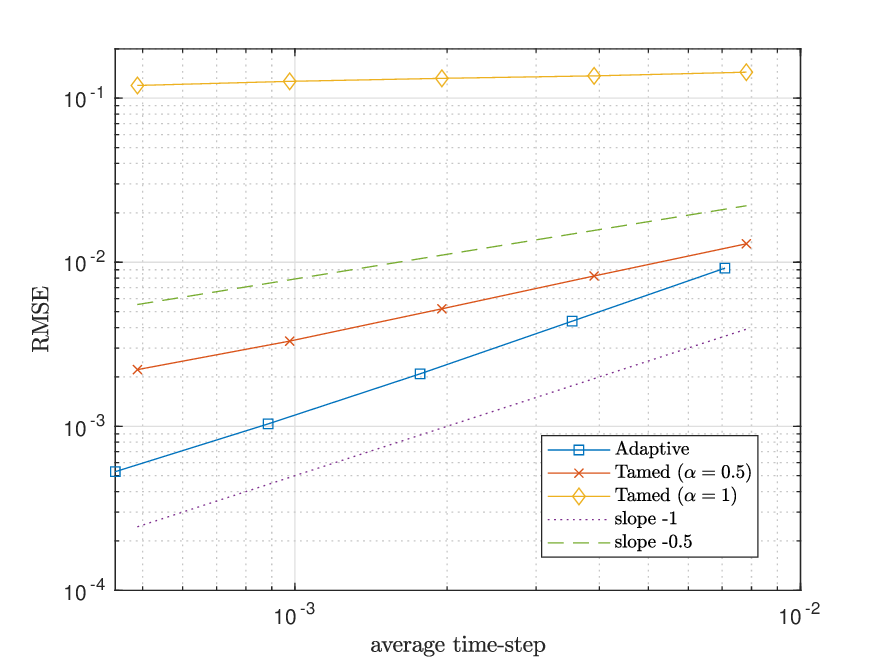}
	\caption{Strong convergence of the adaptive Euler-Maruyama time-discretisation applied to Example 1.}
	\label{fig1}
\end{figure} 

Fig.\ \ref{fig1} plots the root mean-square error (RMSE) against the average time-step size; this is to make the comparison of the tamed Euler-Maruyama scheme and the adaptive method fair.
Here, and for the following examples, ``average time-step'' refers to the harmonic mean of the step-size, i.e.\ we average the number of time-steps over all particles and then take the reciprocal. The number of particles was set to $N=10^4$.

Note that in Fig.\ \ref{fig1} results for the exponent $-1/2$ (i.e., $\alpha = 1/2$) are shown for the tamed scheme in addition to $\alpha = 1$ (chosen in  \cite{FG}), as the latter
yields an even worse convergence rate. Furthermore, it illustrates that the adaptive scheme achieves strong convergence order of $1$ (as the volatility coefficient is constant) in this example compared to the order $1/2$ for the tamed Euler-Maruyama method with $\alpha=1/2$. That we only observe convergence for $\alpha = 1/2$ in the case of tamed schemes, could be argued by the fact that the choice of $\alpha = 1/2$ results in a stronger taming. 

To justify the approximation of the true mean-field term by the arithmetic mean over all particles, we numerically investigate the weak and strong convergence in $N$ of the solution to the particle system, $X^{i,N}$ as in \eqref{eq:PS}, to the solution of the limit McKean-Vlasov SDE, $X^{i}$, as in \eqref{eq:sys}.

To be precise, for the weak convergence we study 
$
  \mathbb{E}[X^{i,N}] - \mathbb{E}[X^{i}],
$
which we expect to be of order $1/N$. To do so, we fix a fine time-grid and compute 
\begin{equation*} 
\text{error} := \left | \frac{1}{K}\sum_{j=1}^{K} \left(\frac{1}{N_l} \sum_{i=1}^{N_l} \hat{X}^{i,N_l,M,(j)}_T - \frac{1}{N_{l+1}} \sum_{i=1}^{N_{l+1}} \hat{X}^{i,N_{l+1},M,(j)}_T \right) \right |, 
\end{equation*} 
i.e., we compute the mean-field term at final time $T$ using two particle systems with a different number of particles (e.g., $N_l$ and $N_{l+1} =2N_{l}$) but using the same time-discretisation for both of them. This procedure will be repeated (independently) $K \in \mathbb{N}$ times (indicated by the superscript $j$) and we finally compute the sample mean of these $K$ realisations. Also, note that $N_l$ Brownian motions are used for the simulation of $\hat{X}^{i,N_l,M, (j)}$, out of the set of $N_{l+1}$ Brownian motions used for the simulation of $\hat{X}^{i,N_{l+1},M, (j)}$. This is done for several levels $l$, where a level refers here to the number of particles. 

For the strong pathwise propagation of chaos result, we compute for a fixed fine time-grid
\begin{equation*} 
\text{RMSE} := \sqrt{\frac{1}{N_l} \sum_{i=1}^{N_l} \left(\hat{X}^{i,N_l,M}_T - \tilde{\hat{X}}^{i,N_{l},M}_T \right)^2 }, 
\end{equation*}  
where the particle system $\tilde{\hat{X}}^{i,N_{l},M}$, $i \in \lbrace 1, \ldots, N_l \rbrace$, is obtained by splitting the set of Brownian motions driving the particle system $\hat{X}^{i,N_{l},M}$, $i \in \lbrace 1, \ldots, N_l \rbrace$, in two sets 
of size $N_l/2$, such that
$(\tilde{\hat{X}}^{i,N_{l},M,(1)})_{i \in \lbrace 1, \ldots, N_l/2 \rbrace}$ is a particle system obtained from the Brownian motions $(W^{i})_{i \in \lbrace 1, \ldots, N_l/2 \rbrace }$ and $(\tilde{\hat{X}}^{i,N_{l},M,(2)})_{i \in \lbrace N_l/2+1, \ldots, N_l \rbrace}$ from $(W^{i})_{i \in \lbrace N_l/2+1,\ldots, N_l \rbrace}$.
In particular, for these two smaller particle systems only $N_l/2$ particles are used to approximate the mean-field term. For our numerical test, we set $M=2^{8}$.
\begin{figure}[!h]
\centering
\begin{subfigure}[b]{0.4\textwidth}
\includegraphics[width=\textwidth]{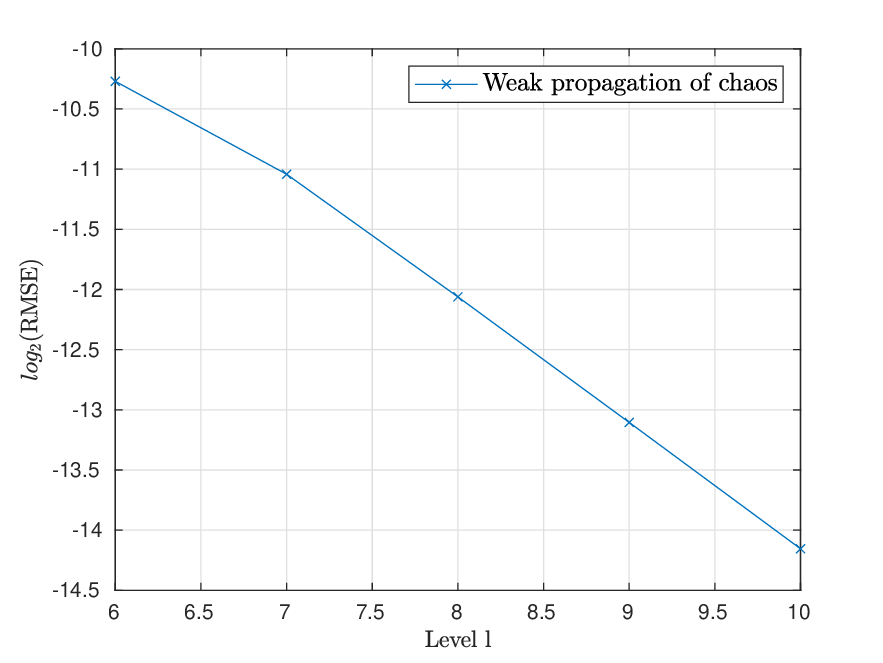}
\caption{}
\label{figWeakPOC}
\end{subfigure}%
\begin{subfigure}[b]{0.4\textwidth}
\includegraphics[width=\textwidth]{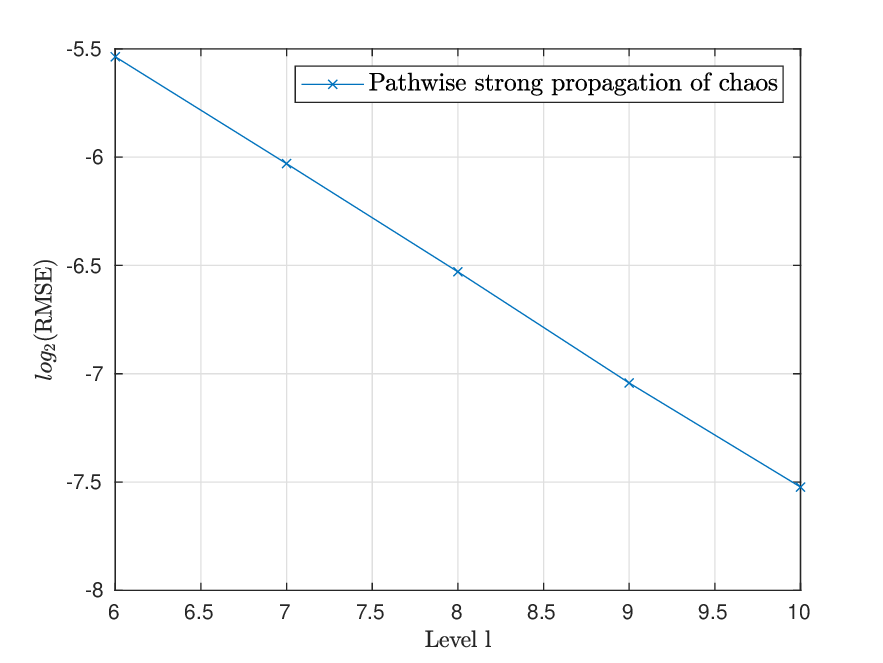}
\caption{}
\label{fig:StrongPOC}
\end{subfigure}
\caption{Example 1: Weak (left) and strong (right) convergence with respect to the number of particles.}
\label{fig:POC}
\end{figure}

We observe in Fig.~\ref{fig:POC} a weak and strong convergence rate of order roughly $1$ (in accordance with the result in \cite{TCS}) and roughly $1/2$, respectively. Recall that \cite[Proposition 3.1]{RES} establishes a {strong} convergence rate (in terms of number of particles) of order $1/4$ for dimensions $d < 4$ only. However, one would expect to have strong convergence of order $1/2$ in $N$, which was indeed proven in, e.g., \cite{DLR}, under a more restrictive set-up than here. The convergence rate in terms of the number of particles can also be improved to $1/2$ if the drift has the special form given as in equation \eqref{MCVlasov}, see \cite{MEL}. 

This illustration is only given for Example 1. For the remaining examples we use the same way of approximating the mean-field term and the results would be similar.

\subsection{Example 2 -- Ginzburg-Landau with added mean-field term}
\label{Example:GL}

As a second example, we consider the one-dimensional McKean-Vlasov SDE 
\begin{equation*} 
    \mathrm{d}X_t = \left( \frac{\sigma^2}{2}X_t -X^3_t + c \mathbb{E}[X_t] \right) \, \mathrm{d}t + \sigma X_t  \, \mathrm{d}W_t, \quad X_0 =1, 
\end{equation*} 
which is a variation of the stochastic Ginzburg-Landau equation (see \cite{DT}) with the addition of a mean-field term. This classical, non-globally Lipschitz SDE satisfies all assumptions listed above (except  (H.\ref{Assum:Ax}(\ref{Assum:Ax6})), which is not needed for the well-posedness) and therefore has a unique strong solution. 


Here, we choose $h^{\delta}(x) = \delta \min(1, |x |^{-2})$,
which can be motivated as follows: as the drift is mainly influenced by the $-X^3_t$ term, the choice $h(x) =  \min(1, |x |^{-2})$ yields $\left \langle  x, b(t,x,\mu) \right \rangle + \frac{1}{2} h(x) | b(t,x,\mu) |^2 < 0$, for large $x$. It is then easy to verify that Assumption (H.\ref{Assum:AxS})
is satisfied for $a=b=1$, $c=2$, and sufficiently large $L_c$ and $L_d$ depending on the model parameters.

In all tests, we take the parameters $\sigma = 1.5$ and $c=0.5$, and $N=10^4$ particles. We set $\alpha=1$ in the tamed scheme. Although, the choice $\alpha=0.5$ would be sufficient to achieve a strong convergence order of $1/2$, we observed that the performance of the tamed scheme with this taming was significantly worse than with $\alpha=1$.   

The strong convergence rate depicted in Fig.\ \ref{fig:test} is roughly of order $1/2$ for both the adaptive and tamed schemes, but the absolute error is smaller by a factor of roughly 10 for the adaptive scheme. 

In addition, we present in Fig.\ \ref{figHISTO} a histogram for $N^{i}_T$, $i \in \lbrace 1, \ldots, N \rbrace$, i.e., the number of adaptive time-steps required for all particles,
for $\delta=2^{-5}$ and $\delta = 2^{-7}$, with bin-width $2^3$ and $2^5$, respectively. The average number of time-steps is approximately $55$ for $\delta=2^{-5}=1/32$ and $218$ for $\delta=2^{-7}=1/128$, consistent with $\mathbb{E}[M_T] = \mathcal{O}(\delta^{-1})$.  
\begin{figure}[H]
\centering
\begin{subfigure}[b]{0.4\textwidth}
\includegraphics[width=\textwidth]{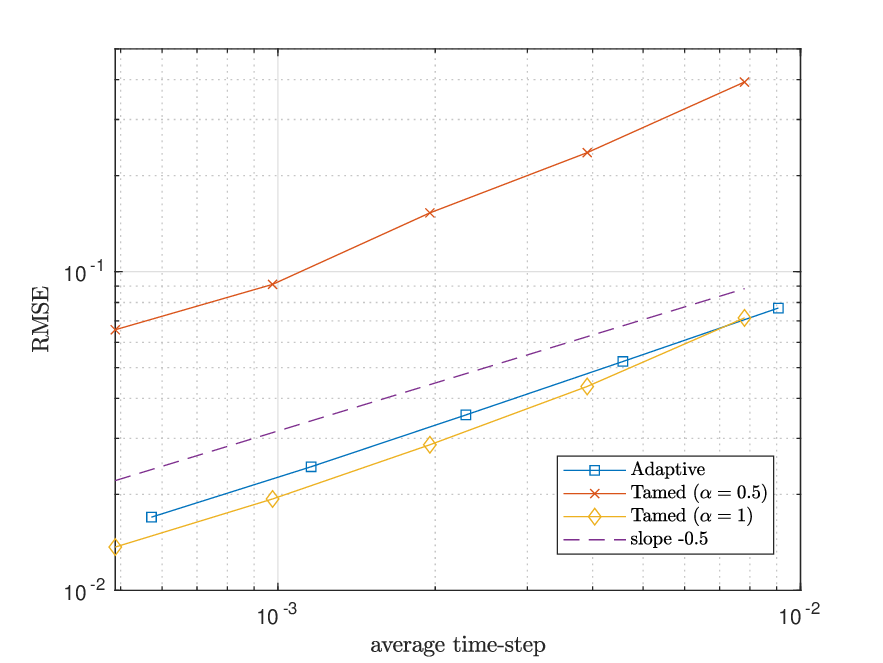}
\caption{}
\label{fig:test}
\end{subfigure}%
\begin{subfigure}[b]{0.4\textwidth}
\includegraphics[width=\textwidth]{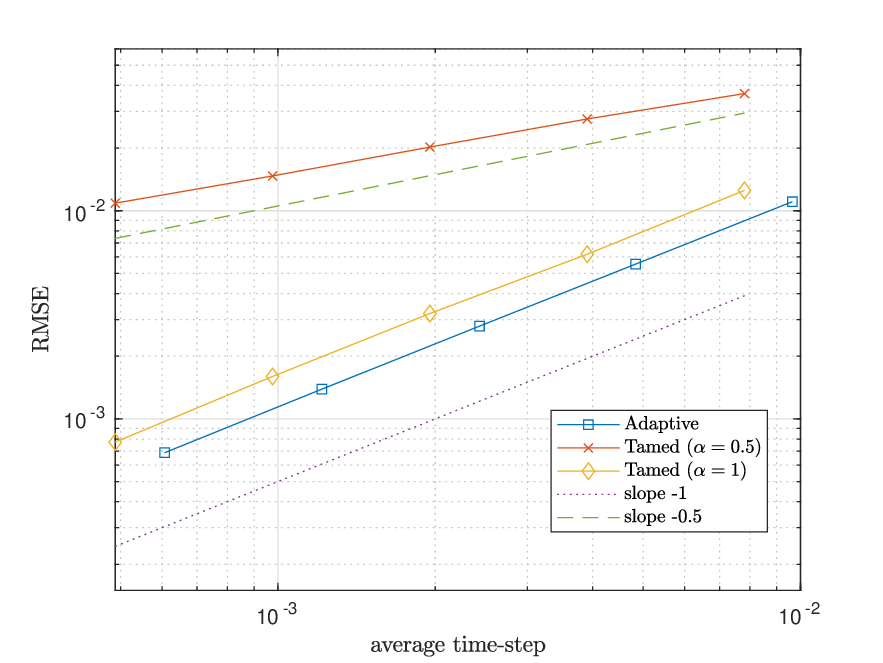}
\caption{}
\label{fig3}
\end{subfigure}
\begin{subfigure}[b]{0.4\textwidth}
\includegraphics[width=\textwidth]{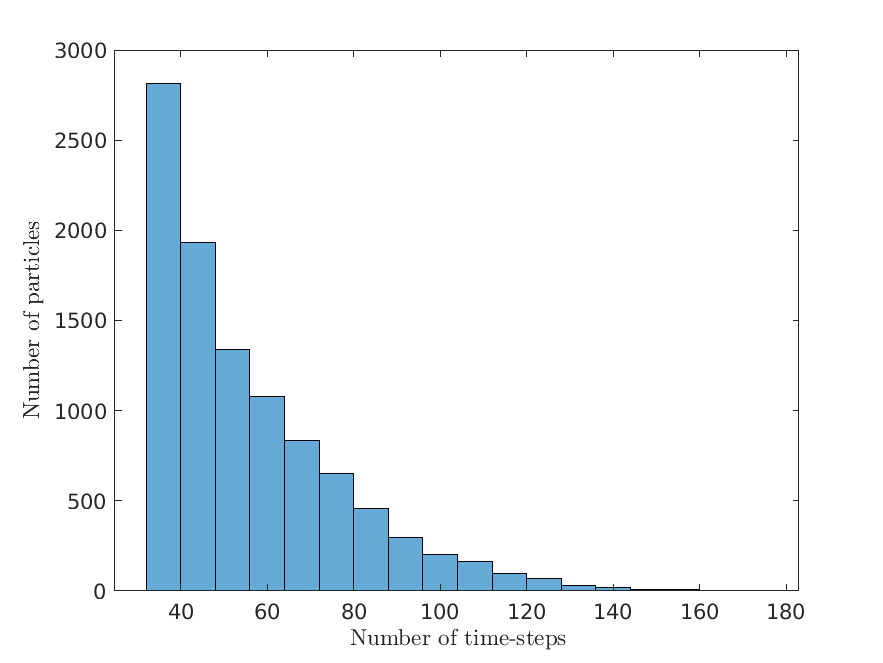}
\caption{}
\label{figHISTO2}
\end{subfigure}
\begin{subfigure}[b]{0.4\textwidth}
\includegraphics[width=\textwidth]{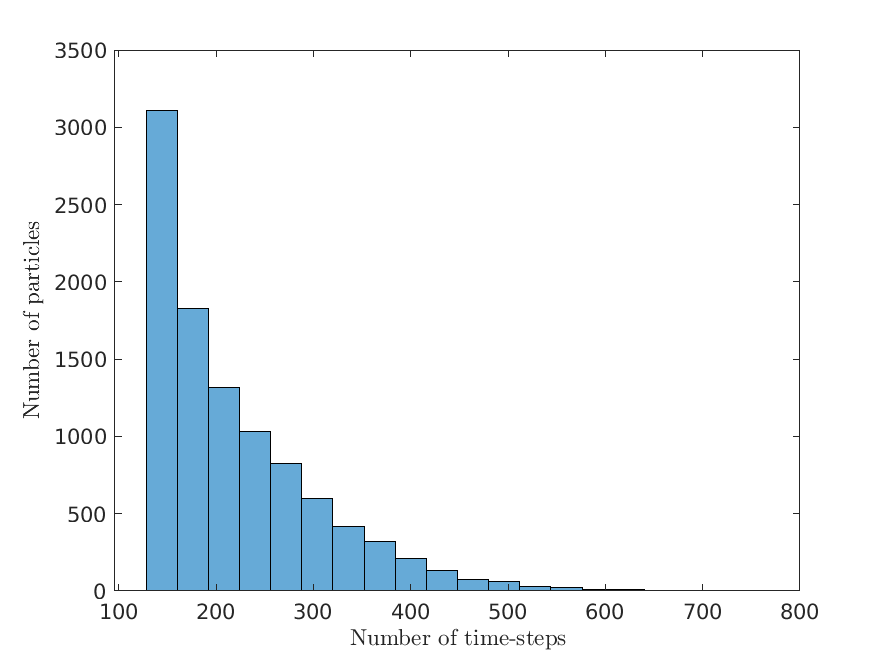}
\caption{}
\label{figHISTO}
\end{subfigure}
\caption{Strong convergence for (a) Example 2 and (b) Example 3. A histogram for the number of time-steps in case of Example 2 for (c) $\delta=2^{-5}$  and (d) $\delta=2^{-7}$.}
\end{figure}


\subsection{Example 3 -- Kuramoto model with added non-Lipschitz drift}
\label{Example:K}

Next, we consider a modification of the one-dimensional Kuramoto model (see \cite{ABVRS}) of the form
\begin{equation*}
    \mathrm{d}X_t^{i,N} = \left( \eta^{i} + X_t^{i,N} -(X_t^{i,N})^3 + \frac{1}{N} \sum_{j=1}^N \sin (X_t^{i,N} - X_t^{j,N}) \right) \, \mathrm{d}t + \sigma \, \mathrm{d}W_t^{i}, \quad X_0^{i,N} = \xi^{i},
\end{equation*}
where $\sigma \in \mathbb{R}$ is a constant and $\eta^{i}$ are i.~i.~d.\ random variables and independent from the set of i.~i.~d.\ random variables $\xi^{i}$ and the Brownian motions $W^i$, for $i \in \lbrace 1, \ldots, N \rbrace$. We set $\sigma=1$, $\xi^{i} \sim \mathcal{U}(0.5,1)$ and $\eta^{i} \sim \mathcal{N}(0,1)$. The term $X_t -X_t^3$, which does not appear in the original model, was added to illustrate the effect of a super-linear  drift term. As in the previous example, we choose $h^{\delta}(x) = \delta \min(1, |x |^{-2})$.

Our numerical tests shown in Fig.~\ref{fig:test} verify the strong convergence of order 1 for the adaptive scheme, which
is due to the constant volatility term in this example. It is only in this special case that we find the adaptive and tamed schemes to behave similarly. The number of particles was set to $N=10^3$.

To achieve a strong convergence rate of 1 for the tamed Euler-Maruyama method, the taming exponent $\alpha$ has to be chosen as 1. For the choice $\alpha = 1/2$ suggested in \cite{RES, SA}, one only achieves the expected strong order of $1/2$. However, note that in the generality of the SDEs considered there, the aim in \cite{RES,SA} was only to prove strong convergence rates of order $1/2$.

\subsection{Example 4 -- FitzHugh-Nagumo model}
\label{Example:NeuroFHN}

As main application, we consider a standard model from neuroscience, the FitzHugh-Nagumo network (see, e.g., \cite{BFFT,BO2,mehri2020}). For $N$ neurons and $P$ different neuron populations, we denote for $i \in \lbrace 1, \ldots, N \rbrace$ by
$p(i) = \alpha$, $\alpha \in \lbrace 1, \ldots, P \rbrace$, the population the $i$-th particle belongs to. For the state vector $(X_t^{i,N})_{t \in [0,T]} =(V_t^{i,N}, w_t^{i,N}, y_t^{i,N})_{t \in [0,T]}$ of neuron $i$, we consider the $3$-dimensional SDE
\begin{align*} 
    \mathrm{d}X_t^{i,N} = f_{\alpha}(t,X_t^{i,N}) \, \mathrm{d}t & + g_{\alpha}(t, X_t^{i,N}) \, \begin{bmatrix} 
           \mathrm{d}W_t^{i} \\ 
           \mathrm{d}W_t^{i,y} 
         \end{bmatrix} \\
         &  + \sum_{\gamma=1}^P \frac{1}{N_\gamma} \sum_{j,p(j)=\gamma} \left( b_{\alpha \gamma}(X_t^{i,N},X_t^{j,N}) \, \mathrm{d}t +  \beta_{\alpha \gamma}(X_t^{i,N},X_t^{j,N}) \,\mathrm{d}W_t^{i,\gamma} \right), 
\end{align*} 
for $i \in \lbrace 1, \ldots, N \rbrace$ and where $N_{\gamma}$ denotes the number of neurons in population $\gamma$. For $\gamma, \alpha \in \lbrace 1, \ldots, P \rbrace$,
$I^{\alpha}(t) := I, \ \forall  t \in [0,T], \ \forall  \alpha,$
for some constant value $I$,
\begin{equation*} 
    f_{\alpha}(t, X_t^{i,N}) = \begin{bmatrix} 
           V_t^{i,N} - \frac{(V_t^{i,N})^3}{3} - w_t^{i,N} + I^{\alpha}(t) \\ 
           c_{\alpha}(V_t^{i,N} + a_{\alpha} - b_{\alpha} w_t^{i,N}) \\ 
           a_r^{\alpha} S_{\alpha}(V_t^{i,N})(1-y_t^{i,N}) - a_d^{\alpha} y_t^{i,N} 
         \end{bmatrix}, \quad
    g_{\alpha}(t, X_t^{i,N}) = \begin{bmatrix} 
           \sigma_{\text{ext}}^{\alpha} & 0 \\ 
           0 & 0 \\ 
          0 & \sigma_{\alpha}^{y}(V_t^{i,N},y_t^{i,N}) 
         \end{bmatrix}, 
\end{equation*} 
and 
\begin{equation*} 
    b_{\alpha \gamma}(X_t^{i,N}, X_t^{j,N}) = \begin{bmatrix} 
           - \bar{J}_{\alpha \gamma}(V_t^{i,N} - V_{\text{rev}}^{\alpha \gamma})y_t^{j,N} \\ 
           0 \\ 
           0 
         \end{bmatrix}, \quad
   \beta_{\alpha \gamma}(X_t^{i,N}, X_t^{j,N}) = \begin{bmatrix} 
           - \sigma_{\alpha \gamma}^{J}(V_t^{i,N} - V_{\text{rev}}^{\alpha \gamma})y_t^{j,N} \\ 
           0 \\ 
           0 
         \end{bmatrix}. 
\end{equation*} 
Moreover, 
\begin{align*}
& S_{\alpha}(V_t^{i,N}) = \frac{T_{\max}^{\alpha}}{1 + e^{- \lambda_{\alpha}(V_t^{i,N} -V_T^{i,N})}},  \qquad  \qquad \qquad  \qquad
\chi(y_t^{i,N}) = \mathbb{I}_{y_t^{i,N} \in (0,1)} \Gamma e^{- \Lambda /(1-(2y_t^{i,N}-1)^2)}, \\
& \sigma_{\alpha}^{y}(V_t^{i,N},y_t^{i,N}) = \sqrt{a_r^{\alpha} S_{\alpha}(V_t^{i,N})(1-y_t^{i,N}) + a_d^{\gamma}y_t^{i,N}}\chi(y_t^{i,N}),
\end{align*}
and the standard one-dimensional Brownian motions $(W^{i},W^{i,y},W^{i,\gamma})$ are mutually independent for all $i \in \lbrace 1, \ldots, N \rbrace$.
For our numerical tests, we use the parameter values listed in \cite{RES} (with $\sigma_{\text{ext}}=0, \sigma^J=0.00002$ and $\sigma_{\text{ext}} =0.5, \sigma^J=0.2$, respectively), where the tamed Euler-Maruyama scheme was numerically tested, and also use the same type of random initial values presented there.

In our numerical experiments, we set $P=1$ and $N=10^3$. Note that the functions $f_{\alpha}$ and $g_{\alpha}$ are uniformly locally Lipschitz continuous with respect to the second variable. Moreover, the functions $\beta_{\alpha \gamma}$ and $b_{\alpha \gamma}$ are globally Lipschitz continuous
and satisfy the growth condition \eqref{eqn:growth}. 
These conditions give the existence of a unique solution to the above limit SDE and propagation of chaos results as shown in \cite{BO2}.

The function $\chi$ is modelled in a way such that it is a bounded Lipschitz function with compact support included in the interval $(0,1)$. Moreover, the value of $y^i$ such that $a_r^{\alpha} S_{\alpha}(V^i)(1-y^i) + a_d^{\gamma}y^i=0$, for $a_r^{\alpha} S_{\alpha}(V^i), a_d^{\gamma} >0$, never belongs to $[0,1]$. Thus, the square-root term in $\sigma_{\alpha}^{y}$ is not zero for $y^i \in [0,1]$. In addition, \cite[Proposition 3.3]{BO2} proves that for $0 \leq y_0^{i} \leq 1$ a.\ s., also $\mathbb{P} \left( \forall t \geq 0, \ 0 \leq y_t^{i} \leq 1 \right) = 1$.
 For the numerical approximation with the Euler-Maruyama scheme, there is a positive probability that $y$ becomes negative. We address this by an adaptation of the full-truncation Euler-Maruyama scheme, i.e., by simply setting this expression to zero, which has been shown to have good stability and accuracy for square-root diffusions
 (see, e.g., \cite{LKD, CR, CR2, CMR}, noting that they consider the more challenging case where the process can get arbitrarily close to the singularity).
 
For the adaptive time-step, we choose 
\begin{equation*} 
     h^{\delta}(x) = \delta \min(T, \gamma | x|^2 / | f_{\alpha}(t,x) |^2), 
\end{equation*} 
for some parameter $\gamma >0$. To motivate this choice, we note that for large $|X_t^{i}|$, $ f_{\alpha}(t, X_t^{i})$ is the driving part of the drift term and therefore we essentially achieve by this choice that $h(x) | b(t,x,\mu) | \approx c | x|^2$ for $h(x)= h^{\delta}(x)/\delta$. 

The strong convergence rates of the tamed and adaptive schemes are depicted for two different volatilities in Fig.\ \ref{fig4} and Fig.\ \ref{fig4a}. We can deduce from Fig.\ \ref{fig4} that the adaptive scheme has a practical strong convergence rate of order 1, as the components of the diffusion matrix are either zero or close to zero. The rate is asymptotically 1/2 for larger volatilities (Fig.\ \ref{fig4a}).
\begin{figure}[!h]
\centering
\begin{subfigure}[b]{0.4\textwidth}
\includegraphics[width=\textwidth]{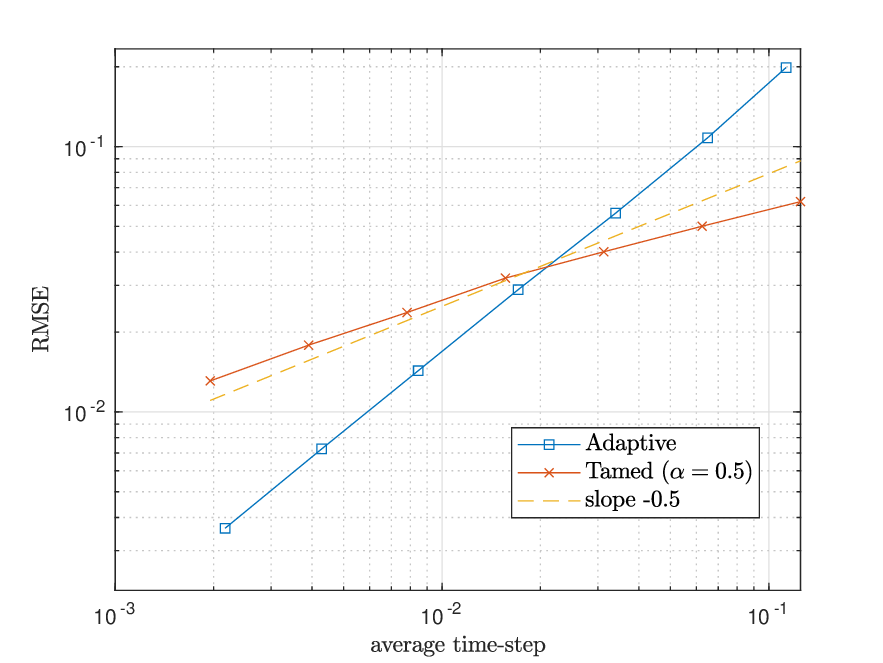}
\caption{$\sigma_{\text{ext}}=0, \sigma^J=0.00002$}
\label{fig4}
\end{subfigure}%
\begin{subfigure}[b]{0.4\textwidth}
\includegraphics[width=\textwidth]{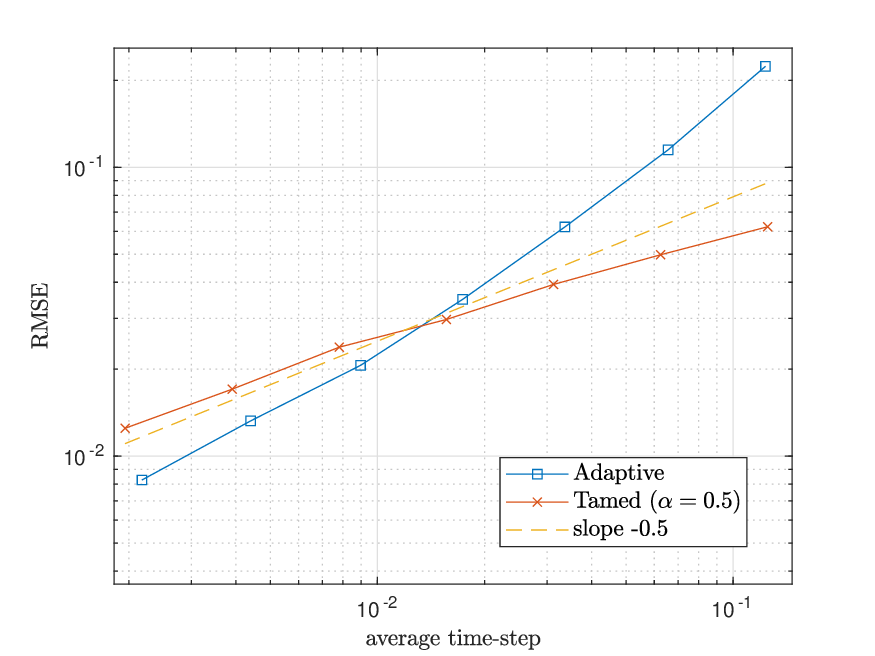}
\caption{$\sigma_{\text{ext}} =0.5, \sigma^J=0.2$}
\label{fig4a}
\end{subfigure}
\caption{Strong convergence for Example 4.}
\end{figure}

To additionally illustrate the importance of using a tamed or an adaptive time-stepping scheme for this model instead of a standard Euler-Maruyama scheme, we demonstrate numerically in Fig.\ \ref{fig:ParticleDivergence} that this scheme yields approximations which start to strongly oscillate and potentially diverge after a finite time (``particle corruption'' from \cite{RES}). For this test, we used $M=2^3$ and $N=10^3$.   
\begin{figure}[!h]
\centering
\begin{subfigure}[b]{0.4\textwidth}
\includegraphics[width=\textwidth]{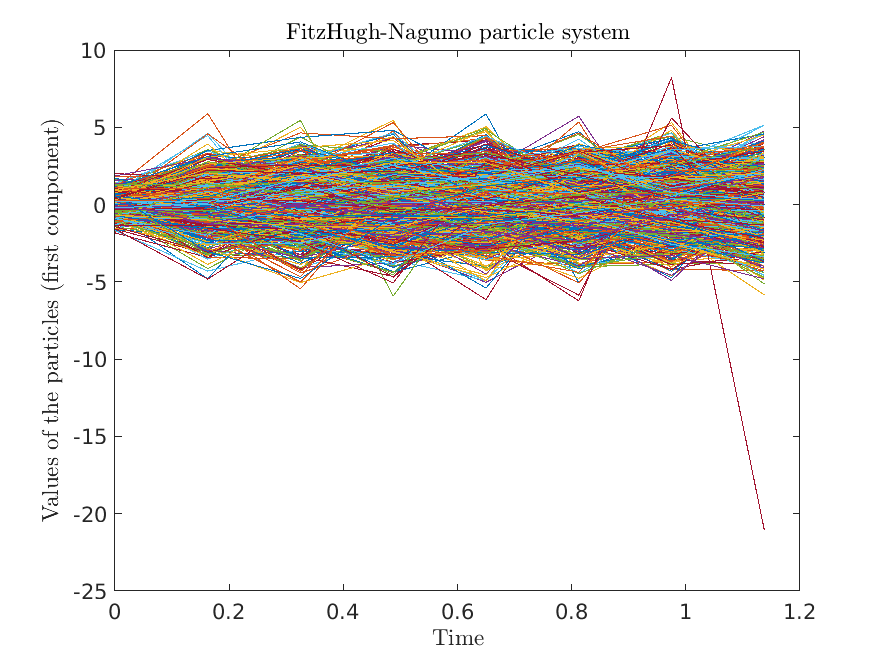}
\caption{}
\label{fig:ParticleDivergence}
\end{subfigure}%
\begin{subfigure}[b]{0.4\textwidth}
\includegraphics[width=\textwidth]{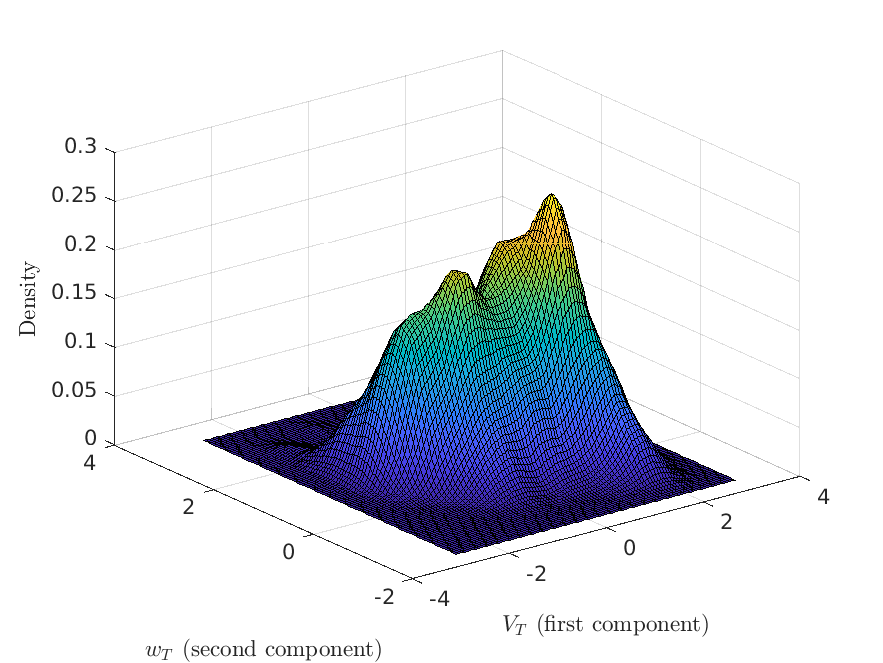}
\caption{}
\label{fig:testdensity}
\end{subfigure}
\caption{
Simulations of the FitzHugh-Nagumo network.
(a) ``Particle corruption'' (in the first component) for a standard Euler-Maruyama scheme.
(b) Approximation of the joint density of $(V_T, w_T)$, i.e., the first two components.
}
\end{figure}

Our  scheme can be used to approximate marginal densities of the above limit equation derived from the FitzHugh-Nagumo network. Fig.\ \ref{fig:testdensity} was obtained by the adaptive time-stepping method with $\delta = 1/2^{8}$, $T=1$ and $N=2 \cdot 10^3$. We used a kernel density approach to obtain the density from the simulated data. 

\subsection{Example 5 -- Milstein scheme}\label{sec:Mil}

To illustrate the performance of the adaptive Milstein scheme, we consider the following particle system 
\begin{align*}
\mathrm{d} X_t^{i,N} = \Bigg( -(X_t^{i,N})^5 + \frac{1}{N} \sum_{j=1}^{N} X_t^{j,N}  \Bigg) \, \mathrm{d}t + X_t^{i,N} \, \mathrm{d}W_t^{i}, \quad X_0^{i,N} =1,
\end{align*}
and use as benchmark the tamed Milstein scheme 
\begin{align*} 
    Y_{t_{n+1}}^{i,N,M} =  Y_{t_{n}}^{i,N,M} + \frac{- (Y_{t_n}^{i,N,M})^5 + \frac{1}{N}\sum_{j=1}^N Y_{t_n}^{j,N,M}}{1+h \left | - (Y_{t_n}^{i,N,M})^5 + \frac{1}{N} \sum_{j=1}^N Y_{t_n}^{j,N,M} \right |}h  + Y_{t_{n}}^{i,N,M} \Delta W_{t_n}^{i} + \frac{1}{2} Y_{t_{n}}^{i,N,M} ((\Delta W_{t_n}^i)^2 -h), 
\end{align*} 
where $Y^{i,N,M}_0 =1$, for all $i \in \lbrace 1, \ldots, N \rbrace$. 
Fig.\ \ref{figMilstein} illustrates the predicted strong convergence of order $1$ for both the tamed and adaptive Milstein scheme. For the adaptive time-step function, we choose $h^{\delta}(x)= \delta \min(1,| x|^{-4})$. The number of particles was set to $N=10^4$.
\begin{figure}[H]
	\centering
  \includegraphics[width=0.4\textwidth]{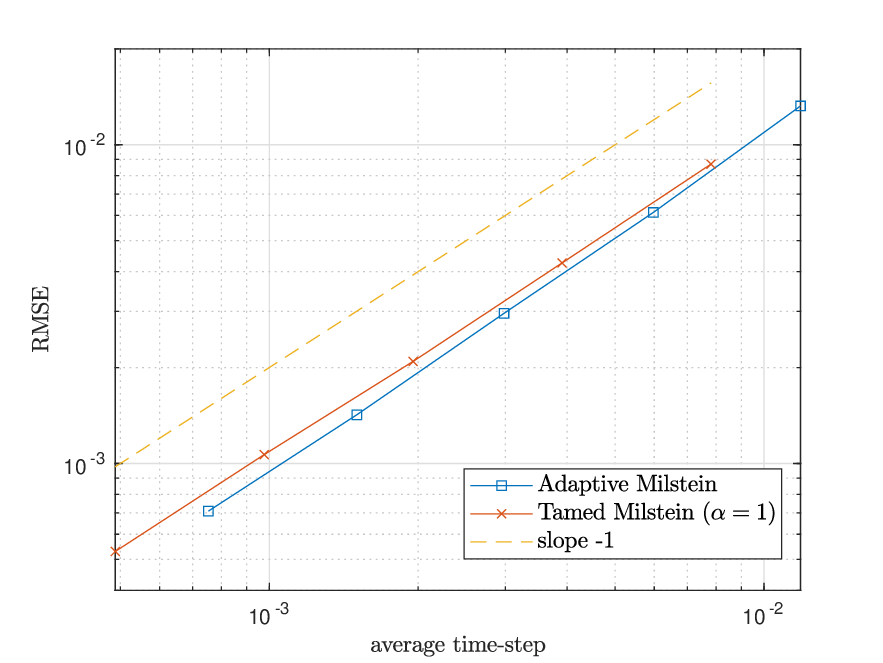}
	\caption{Strong convergence of the Milstein schemes.}
	\label{figMilstein}
\end{figure} 

\section{Proofs of main results}
\label{SECPro}

We start by giving stability 
results for the processes $\hat{X}^{i,N}$ and $\tilde{X}^{i,N}$ defined in (\ref{AdaptiveMcKean}) and (\ref{eq:TDPS}), respectively, for $i \in \lbrace 1, \ldots, N \rbrace$.
Note that generic constants used in the proofs only depend on $p, T$, the moment bounds for the initial data and the constants appearing in the assumptions on $b$, $\sigma$ and $h$, but are independent of $N$ and $\delta$.

\subsection{Moment stability for super-linear drift coefficient} 

In this section, our main result on stability is first proven for $p \geq 4$. Then, H\"{o}lder's inequality allows to deduce the claim for $ 0 < p < 4$.

\begin{prop}\label{prop1}
Let $p>0$ and $\xi \in L^p_0(\mathbb{R}^d)$. Let Assumption (H.\ref{Assum:Ax}) hold, and suppose the time-step function satisfies (H.\ref{Assum:AxS}(\ref{Assum:AxS2})). Then, $T$ is almost surely attainable, i.e. 
$\mathbb{P}(\exists \ M(\omega) < \infty, \text{ s.t. } t_{M(\omega)} \geq T)=1$.
Moreover, there exists a constant $C>0$ such that 
\begin{equation*}
    \max_{i \in \lbrace 1, \ldots, N \rbrace} \mathbb{E}\left[\sup_{t \in [0,T]} |  \hat{X}_{t}^{i,N} |^p \right] \ \lor \ \max_{i \in \lbrace 1, \ldots, N \rbrace} \mathbb{E}\left[\sup_{t \in [0,T]} |  \tilde{X}_{t}^{i,N} |^p \right]  \leq C.
\end{equation*}
\end{prop}

\begin{proof}
We follow a strategy similar to that in \cite[Theorem 1]{FG} for non-Lipschitz standard SDEs and adapt the necessary steps for the measure dependence. Let $K \geq 1$ be an integer, then we define a
$K$-truncated process by $X_0^{K,i,N}:= P_K(X_0^{i})$ and
\begin{equation}\label{eq:eq1}
  \hat{X}_{t_{n +1}}^{K,i,N} =  P_K(\hat{X}_{t_n}^{K,i,N} +  b(t_n, \hat{X}_{t_n}^{K,i,N}, \mu_{t_n}^{\hat{\boldsymbol{X}}^N}) h_n^{\min} +\sigma(t_n, \hat{X}_{t_n}^{K,i,N}, \mu_{t_n}^{\hat{\boldsymbol{X}}^N}) \Delta W_{t_n}^i),
\end{equation}
where, $\mu_{t_n}^{\hat{\boldsymbol{X}}^N}$ is the empirical measure (defined as above in (\ref{eq:empMeasure}), but using the $N$ truncated processes $\hat{X}_{t_n}^{K,i,N}$),
$P_K(Y) := \min(1,K/ | Y |)Y$, and consequently $|  \hat{X}_{t_{n}}^{K,i,N}  | \leq K$, for all $i \in \lbrace 1, \ldots, N \rbrace$ and $n \geq 0$. The continuous time interpolation for $\hat{X}_t^{K,i,N}$ is defined analogously to above (see (\ref{AdaptiveMcKean})).

The reason for introducing $\hat{X}^{K,i,N}$ is that one can guarantee that the time-step function has a strictly positive lower bound, i.e., $\inf_{| x| \leq K} h(x) >0$. Hence, $T$ is attainable. Showing that the moments of the $K$-truncated processes are bounded by some constant $C>0$ independent of $K$ allows to send $K$ to infinity and the result will then follow in combination with the monotone convergence theorem. 

The following analysis is performed for a fixed particle.
We get, using (\ref{eq:eq1}), 
\begin{align}\label{eq:eq2}
| \hat{X}_{t_{n +1}}^{K,i,N} |^2 &\leq  | \hat{X}_{t_{n}}^{K,i,N} |^2 + 2 h_n^{\min} \left( \left \langle \hat{X}_{t_{n}}^{K,i,N}, b(t_n, \hat{X}_{t_n}^{K,i,N}, \mu_{t_n}^{\hat{\boldsymbol{X}}^N}) \right \rangle  + \frac{1}{2} h_n^{\min} | b(t_n, \hat{X}_{t_n}^{K,i,N}, \mu_{t_n}^{\hat{\boldsymbol{X}}^N}) |^2 \right ) \nonumber \\
& \quad +2 \left \langle \phi(\hat{X}_{t_n}^{K,i,N},\mu_{t_n}^{\hat{\boldsymbol{X}}^N}) , \sigma(t_n, \hat{X}_{t_n}^{K,i,N}, \mu_{t_n}^{\hat{\boldsymbol{X}}^N}) \Delta W_{t_n}^i \right \rangle + | \sigma(t_n, \hat{X}_{t_n}^{K,i,N}, \mu_{t_n}^{\hat{\boldsymbol{X}}^N}) \Delta W_{t_n}^i |^2,
\end{align}
where 
\begin{equation*}
\phi(\hat{X}_{t_n}^{K,i,N},\mu_{t_n}^{\hat{\boldsymbol{X}}^N}) := \hat{X}_{t_n}^{K,i,N} + h_n^{\min} b(t_n, \hat{X}_{t_n}^{K,i,N},\mu_{t_n}^{\hat{\boldsymbol{X}}^N}).
\end{equation*}
Using Assumption (H.\ref{Assum:AxS}(\ref{Assum:AxS2})) for the time-step function, we get for (\ref{eq:eq2})
\begin{align}\label{eq:eq3}
    | \hat{X}_{t_{n +1}}^{K,i,N} |^2 &\leq  | \hat{X}_{t_{n}}^{K,i,N} |^2 + 2 L_c  | \hat{X}_{t_{n}}^{K,i,N} |^2 h_n^{\min}+ 2 L_d h_n^{\min} \nonumber \\
    & \quad +2 \left \langle \phi(\hat{X}_{t_n}^{K,i,N},\mu_{t_n}^{\hat{\boldsymbol{X}}^N}) , \sigma(t_n, \hat{X}_{t_n}^{K,i,N}, \mu_{t_n}^{\hat{\boldsymbol{X}}^N}) \Delta W_{t_n}^i \right \rangle + | \sigma(t_n, \hat{X}_{t_n}^{K,i,N},\mu_{t_n}^{\hat{\boldsymbol{X}}^N}) \Delta W_{t_n}^i |^2.
\end{align}
Summing (\ref{eq:eq3}) over multiple time-steps, adding the contribution from $\underline{t}$ to $t$ and then using Jensen's inequality, results in 
\allowdisplaybreaks
\begin{align}\label{eq:eqStab}
    | \hat{X}_{t}^{K,i,N} |^p &\leq C_p \Bigg[ | X_0^{K,i,N} |^p + \left( 2 L_c \int_0^{t} | \bar{X}_s^{K,i,N} |^2 \, \mathrm{d}s \right)^{p/2} + (2 t L_d)^{p/2} \nonumber \\
    & \qquad + \Big| 2 \int_{0}^{\underline{t}} \Big \langle \phi(\bar{X}_{s}^{K,i,N},\mu_{s}^{\bar{\boldsymbol{X}}^N}) , \sigma(\underline{s}, \bar{X}_{s}^{K,i,N}, \mu_{s}^{\bar{\boldsymbol{X}}^N}) \, \mathrm{d} W_s^{i} \Big \rangle \Big|^{p/2}  \nonumber \\
    & \qquad + \left( \sum_{k=0}^{n_t-1} |\sigma(t_k, \bar{X}_{t_k}^{K,i,N}, \mu_{t_k}^{\bar{\boldsymbol{X}}^N}) \Delta W_{t_k}^i  |^2 \right)^{p/2} \nonumber \\
    & \qquad + \Big| 2 \Big\langle \bar{X}_t^{K,i,N} + b(\underline{t},\bar{X}_{t}^{K,i,N},\mu_{t_n}^{\bar{\boldsymbol{X}}^N})(t - \underline{t}), \sigma(\underline{t}, \bar{X}_{t}^{K,i,N}, \mu_{t}^{\bar{\boldsymbol{X}}^N})(W_{t}^i -W_{\underline{t}}^i) \Big \rangle \Big|^{p/2}  \nonumber  \\ 
    & \qquad 
    + |\sigma(\underline{t}, \bar{X}_{t}^{K,i,N}, \mu_{t}^{\bar{\boldsymbol{X}}^N}) ( W_{t}^i -W_{\underline{t}}^i) |^p  \Bigg],
\end{align}
where we recall $n_t= \max \lbrace n: t_n \leq t \rbrace$ and $C_p>0$ is a constant depending on $p$ only. Taking the supremum (over $t$) and the expectation in (\ref{eq:eqStab}) gives
\begin{align*}
\mathbb{E} \left[ \sup_{t \in [0,T]} | \hat{X}_{t}^{K,i,N} |^p \right] \leq C_p \sum_{j=1}^{5} \Pi_j,
\end{align*}
where 
\begin{align*}
& \Pi_1 := \mathbb{E} \left[ | X_0^{K,i,N} |^p \right] + \mathbb{E} \left[ \left( 2 L_c \int_0^{t} | \bar{X}_s^{K,i,N} |^2 \, \mathrm{d}s \right)^{p/2} \right] + (2 t L_d)^{p/2}, \\
& \Pi_2 := \mathbb{E} \left[ \sup_{s \in [0,\underline{t}]} \Big| 2 \int_{0}^{s} \Big \langle \phi(\bar{X}_{u}^{K,i,N},\mu_{u}^{\bar{\boldsymbol{X}}^N}), \sigma(\underline{u}, \bar{X}_{u}^{K,i,N}, \mu_{u}^{\bar{\boldsymbol{X}}^N}) \, \mathrm{d} W_u^{i} \Big \rangle \Big|^{p/2} \right], \\
& \Pi_3:= \mathbb{E} \left[ \left( \sum_{k=0}^{n_t-1} |\sigma(t_k, \bar{X}_{t_k}^{K,i,N}, \mu_{t_k}^{\bar{\boldsymbol{X}}^N}) \Delta W_{t_k}^i  |^2 \right)^{p/2} \right], \\
& \Pi_4 := \mathbb{E} \left[ \sup_{s \in [0,t]} \Big| 2 \Big\langle \bar{X}_s^{K,i,N} + b(\underline{s},\bar{X}_{s}^{K,i,N},\mu_{s}^{\bar{\boldsymbol{X}}^N}) (s - \underline{s}),\sigma(\underline{s}, \bar{X}_{s}^{K,i,N},\mu_{s}^{\bar{\boldsymbol{X}}^N}) (W_{s}^i -W_{\underline{s}}^i) \Big \rangle \Big|^{p/2} \right], \\
& \Pi_5:= \mathbb{E} \left[ \sup_{s \in [0,t]} |\sigma(\underline{s}, \bar{X}_{s}^{K,i,N},\mu_{s}^{\bar{\boldsymbol{X}}^N}) (W_{s}^i -W_{\underline{s}}^i) |^p \right].
\end{align*}

Note that Assumption (H.\ref{Assum:AxS}(\ref{Assum:AxS2})) implies that 
\allowdisplaybreaks
\begin{align*}
   | \phi(\bar{X}_{s}^{K,i,N},\mu_{s}^{\bar{\boldsymbol{X}}^N})|^2 & = | \bar{X}_{s}^{K,i,N} + h_{n_s}^{\min} b(\underline{s},\bar{X}_{s}^{K,i,N}, \mu_{s}^{\bar{\boldsymbol{X}}^N}) |^2  \\ 
    & \leq | \bar{X}_{s}^{K,i,N} |^2 +2h_{n_s}^{\min}(L_c| \bar{X}_{s}^{K,i,N} |^2 +L_d)  \\
    & \leq (1 + 2L_cT) |\bar{X}_{s}^{K,i,N} |^2 + 2L_dT,
\end{align*}
and consequently
\begin{align}\label{est1}
     & | \phi(\bar{X}_{s}^{K,i,N},\mu_{s}^{\bar{\boldsymbol{X}}^N}) |^{p/2} \leq 2^{p/4 -1} \left((1+2L_cT)^{p/4} |\bar{X}_{s}^{K,i,N} |^{p/2} +(2TL_d)^{p/4} \right).
\end{align}
Further, we have that
\begin{align}\label{est2}
 | \bar{X}_t^{K,i,N} + b(\underline{t},\bar{X}_{t}^{K,i,N},\mu_{t}^{\bar{\boldsymbol{X}}^N})(t - \underline{t})|^{p/2}  \leq C_{p,T} \left( |\bar{X}_t^{K,i,N}|^{p/2} + 1 \right),
\end{align}
where we employed (H.\ref{Assum:AxS}(\ref{Assum:AxS2})) and $C_{p,T}> 0$ is a constant depending on $p$, $T$ and the constants appearing in the assumptions on the coefficient $b$ and in (H.\ref{Assum:AxS}(\ref{Assum:AxS2})).

Using this, we note that each of the terms $\Pi_1, \ldots, \Pi_5$ can be estimated similarly to the proof of \cite[Theorem 1]{FG}, employing additionally the bound
\begin{equation*}
\| \sigma(t,\bar{X}_{t}^{K,i,N}, \mu_{t}^{\bar{\boldsymbol{X}}^N}) \| \leq C\left(|\bar{X}_{t}^{K,i,N} | + 1 \right),
\end{equation*}
for $C > 0$ (depending on the Lipschitz constant for $\sigma$ and the constant appearing in  (H.\ref{Assum:Ax}(\ref{Assum:Ax6}))). To be precise, having estimates (\ref{est1}) and (\ref{est2}) in mind, along with the fact that the initial data has finite moments up to order $p$, each $\Pi_j$, $j \in \lbrace 1, \ldots, 5 \rbrace$, can be bounded by 
\begin{align*}
\Pi_j \leq C_{p,T} + C_{p,T} \int_{0}^{t} \mathbb{E} \left[ \sup_{u \in [0,s]} | \hat{X}_{u}^{K,i,N} |^p \right] \, \mathrm{d}s,
\end{align*} 
where $C_{p,T}>0$ is a constant independent of $K$, and $N$ and the particle index $i$. 
This allows to conclude the stability of the processes $\hat{X}^{i,N,K}$ using Gronwall's inequality. 

It remains to show the attainability of $T$ for the non-truncated particle system. The stability of $\hat{X}^{i,N}$ will then be a consequence of the monotone convergence theorem.
We have, for any $\omega \in \Omega$, $\max_{j \in \lbrace 1, \ldots, N \rbrace} \hat{X}_{t}^{K,j,N}(\omega) = \max_{j \in \lbrace 1, \ldots, N \rbrace} \hat{X}_{t}^{j,N}(\omega)$ for all $t \in [0,T]$ if and only if the inequality $\max_{j \in \lbrace 1, \ldots, N \rbrace} \sup_{t \in [0,T]} | \hat{X}_{t}^{j,N}(\omega) | \leq K$ holds. Employing Markov's inequality, we derive
\begin{align*}
\mathbb{P}\left(\max_{j \in \lbrace 1, \ldots, N \rbrace} \sup_{t \in [0,T]} | \hat{X}_{t}^{j,N} | <  K \right) & = \mathbb{P}\left( \max_{j \in \lbrace 1, \ldots, N \rbrace} \sup_{t \in [0,T]} | \hat{X}_{t}^{K,j,N} | <  K \right)  \\
& \leq 1 - N \frac{\mathbb{E} \left[ \sup_{t \in [0,T]} |\hat{X}_{t}^{K,i,N} |^4  \right] }{K^4}  \to 1,
\end{align*}
as $K \to \infty$, which implies that $\max_{j \in \lbrace 1, \ldots, N \rbrace} \sup_{t \in [0,T]} | \hat{X}_{t}^{j,N}(\omega)| < \infty$ and hence $T$ is attainable. The remaining arguments follow along the same lines as in \cite[Theorem 1, step 4]{FG}. The moment stability of  $\tilde{X}^{i,N}$ can be proven in a similar manner by using Assumption (H.\ref{Assum:Ax}(\ref{Assum:Ax6})).  
\end{proof}

\subsection{Convergence of Scheme 1} \label{Sec:AScheme} 

\begin{prop}\label{prop2}
Let $p>0$ and $X_0 \in L^k_0(\mathbb{R}^d)$, for a sufficiently large $k$ (depending on $p$). If the SDE satisfies Assumption (H.\ref{Assum:Ax}), and the time-step function satisfies (H.\ref{Assum:AxS}(\ref{Assum:AxS2}))--(H.\ref{Assum:AxS}(\ref{Assum:AxS3})), then there exists a constant $C>0$ such that 
\begin{equation*}
  \max_{i \in \lbrace 1, \ldots, N \rbrace} \mathbb{E}\left[ \sup_{t \in [0,T]} | \hat{X}_t^{i,N} - X_t^{i,N} |^p \right] \leq C \delta^{p/2},
\end{equation*}
where $\hat{X}^{i,N}$ and $X^{i,N}$ are defined in (\ref{AdaptiveMcKean}) and (\ref{eq:PS}), respectively. 
\end{prop}

\begin{proof}
Define $e_t^{(i)}:=\hat{X}_t^{i,N} -X_t^{i,N}$, then we get
\begin{align*}
    \mathrm{d}e_t^{(i)} = \left( b(\underline{t}, \bar{X}_{t}^{i,N}, \mu_{t}^{\bar{\boldsymbol{X}}^N}) - b(t, X_{t}^{i,N}, \mu_{t}^{\boldsymbol{X}^N}) \right) \, \mathrm{d}t + \left( \sigma(\underline{t}, \bar{X}_{t}^{i,N}, \mu_{t}^{\bar{\boldsymbol{X}}^N}) -\sigma(t, X_{t}^{i,N}, \mu_{t}^{\boldsymbol{X}^N}) \right) \, \mathrm{d}W_t^i.
\end{align*}
It\^{o}'s formula implies
\begin{align*}
   | e_t^{(i)} |^2 &=  2  \int_{0}^t \left \langle e_s^{(i)}, b(\underline{s}, \bar{X}_{s}^{i,N}, \mu_{s}^{\bar{\boldsymbol{X}}^N}) - b(s, X_{s}^{i,N}, \mu_{s}^{\boldsymbol{X}^N}) \right \rangle \, \mathrm{d}s \\
  & \quad +   \int_{0}^t \| \sigma(\underline{s}, \bar{X}_{s}^{i,N}, \mu_{s}^{\bar{\boldsymbol{X}}^N}) -\sigma(s, X_{s}^{i,N}, \mu_{s}^{\boldsymbol{X}^N}) \|^2 \, \mathrm{d}s  \\
    & \quad 
    +2  \int_{0}^t \left \langle e_s^{(i)}, \left( \sigma(\underline{s}, \bar{X}_{s}^{i,N}, \mu_{s}^{\bar{\boldsymbol{X}}^N}) -\sigma(s, X_{s}^{i,N}, \mu_{s}^{\boldsymbol{X}^N}) \right) \, \mathrm{d}W_s^{i} \right \rangle.
\end{align*}
Note that 
\allowdisplaybreaks
\begin{align*}
    &\left \langle  e_s^{(i)}, b(\underline{s}, \bar{X}_{s}^{i,N}, \mu_{s}^{\bar{\boldsymbol{X}}^N}) - b(s, X_{s}^{i,N}, \mu_{s}^{\boldsymbol{X}^N}) \right  \rangle 
    = \left \langle e_s^{(i)}, b(\underline{s}, \bar{X}_{s}^{i,N}, \mu_{s}^{\bar{\boldsymbol{X}}^N}) -  b(s, \bar{X}_{s}^{i,N}, \mu_{s}^{\bar{\boldsymbol{X}}^N}) \right \rangle \\
    & \qquad\qquad\qquad + \left \langle e_s^{(i)}, b(s, \bar{X}_{s}^{i,N}, \mu_{s}^{\bar{\boldsymbol{X}}^N}) -  b(s, \hat{X}_{s}^{i,N}, \mu_{s}^{\bar{\boldsymbol{X}}^N}) \right \rangle
     + \left \langle e_s^{(i)}, b(s, \hat{X}_{s}^{i,N}, \mu_{s}^{\bar{\boldsymbol{X}}^N}) -  b(s, X_{s}^{i,N}, \mu_{s}^{\bar{\boldsymbol{X}}^N}) \right \rangle \\
    & \qquad\qquad\qquad +\left \langle e_s^{(i)}, b(s, X_{s}^{i,N}, \mu_{s}^{\bar{\boldsymbol{X}}^N}) -  b(s, X_{s}^{i,N}, \mu_{s}^{\hat{\boldsymbol{X}}^N}) \right \rangle
     +\left \langle e_s^{(i)}, b(s, X_{s}^{i,N}, \mu_{s}^{\hat{\boldsymbol{X}}^N}) -  b(s, X_{s}^{i,N}, \mu_{s}^{\boldsymbol{X}^N}) \right \rangle.
\end{align*}
We estimate each of these terms separately:
The first term can be estimated by time-H\"{o}lder continuity,
\begin{align*}
\left \langle e_s^{(i)}, b(\underline{s}, \bar{X}_{s}^{i,N}, \mu_{s}^{\bar{\boldsymbol{X}}^N}) -  b(s, \bar{X}_{s}^{i,N}, \mu_{s}^{\bar{\boldsymbol{X}}^N}) \right \rangle  \leq C |e_s^{(i)}| (s-\underline{s})^{1/2} \leq \frac{1}{2} | e_s^{(i)} |^2 + \frac{C^2}{2}(s-\underline{s}).
\end{align*}
For the second term, we get due to the polynomial growth condition on $b$
\begin{eqnarray*}
\left \langle e_s^{(i)}, b(s, \bar{X}_{s}^{i,N}, \mu_{s}^{\bar{\boldsymbol{X}}^N}) -  b(s, \hat{X}_{s}^{i,N}, \mu_{s}^{\bar{\boldsymbol{X}}^N}) \right \rangle
    &\leq& |e_s^{(i)} | \tilde{L}(\hat{X}_{s}^{i,N},\bar{X}_{s}^{i,N})  | \bar{X}_{s}^{i,N} - \hat{X}_{s}^{i,N} | \\
    &\leq& \frac{1}{2} | e_s^{(i)} |^2 + \frac{1}{2}\tilde{L}(\hat{X}_{s}^{i,N},\bar{X}_{s}^{i,N})^2  | \bar{X}_{s}^{i,N} - \hat{X}_{s}^{i,N} |^2,
\end{eqnarray*}
where 
    $\tilde{L}(x,y):=L(1 + |x |^q + | y |^q)$. 
For the third term, we have due to the one-sided Lipschitz condition
\begin{align*}
   & \left \langle e_s^{(i)}, b(s, \hat{X}_{s}^{i,N}, \mu_{s}^{\bar{\boldsymbol{X}}^N}) - b(s, X_{s}^{i,N},\mu_{s}^{\bar{\boldsymbol{X}}^N}) \right \rangle \leq L_b | e_s^{(i)} |^2.
\end{align*}
The fourth term can be estimated using bounds for the Wasserstein metric, i.e., we obtain 
\begin{align*}
\left \langle e_s^{(i)}, b(s, X_{s}^{i,N}, \mu_{s}^{\bar{\boldsymbol{X}}^N}) -  b(s, X_{s}^{i,N}, \mu_{s}^{\hat{\boldsymbol{X}}^N}) \right \rangle &\leq L | e_s^{(i)} | \mathcal{W}_2(\mu_{s}^{\bar{\boldsymbol{X}}^N} ,\mu_{s}^{\hat{\boldsymbol{X}}^N})  \\
    & \leq L | e_s^{(i)} | \frac{1}{\sqrt{N}} \left( \sum_{j=1}^N | \bar{X}_{s}^{j,N} - \hat{X}_{s}^{j,N}  |^2 \right)^{1/2} \\
    & \leq \frac{1}{2} | e_s^{(i)} |^2 + L^2\frac{1}{2N} \sum_{j=1}^N | \bar{X}_{s}^{j,N} - \hat{X}_{s}^{j,N}  |^2.
\end{align*}
Similarly, for the last term, we derive the following estimate:
\begin{align*}
    \left \langle e_s^{(i)}, b(s, X_{s}^{i,N}, \mu_{s}^{\hat{\boldsymbol{X}}^N}) -  b(s, X_{s}^{i,N}, \mu_{s}^{\boldsymbol{X}^N}) \right \rangle 
     \leq \frac{1}{2} | e_s^{(i)} |^2 + L^2 \frac{1}{2N} \sum_{j=1}^N | e_s^{(j)} |^2. 
\end{align*}
Further, due to the Lipschitz and time-H\"{o}lder continuity conditions on $\sigma$, we get 
\allowdisplaybreaks
\begin{align*}
   & \| \sigma(\underline{s}, \bar{X}_{s}^{i,N}, \mu_{s}^{\bar{\boldsymbol{X}}^N}) -\sigma(s, X_{s}^{i,N}, \mu_{s}^{\boldsymbol{X}^N}) \|^2 \\
   & \leq 2 \| \sigma(\underline{s}, \bar{X}_{s}^{i,N}, \mu_{s}^{\bar{\boldsymbol{X}}^N}) -\sigma(s, \bar{X}_{s}^{i,N}, \mu_{s}^{\bar{\boldsymbol{X}}^N}) \|^2 + 2 \| \sigma(s, \bar{X}_{s}^{i,N}, \mu_{s}^{\bar{\boldsymbol{X}}^N}) -\sigma(s, X_{s}^{i,N}, \mu_{s}^{\boldsymbol{X}^N}) \|^2 \\
   & \leq 2C^2(s - \underline{s})+ 2L^2 ( | \bar{X}_{s}^{i,N} - X_{s}^{i,N}  |^2 + \mathcal{W}_2(\mu_{s}^{\bar{\boldsymbol{X}}^N},  \mu_{s}^{\boldsymbol{X}^N}) )^2 \\
    & \leq 2C^2(s - \underline{s}) + 4L^2 | e_s^{(i)} |^2 + 4L^2 | \hat{X}_{s}^{i,N} - \bar{X}_{s}^{i,N} |^2 + 2L^2  \mathcal{W}_2^{2}(\mu_{s}^{\bar{\boldsymbol{X}}^N},  \mu_{s}^{\boldsymbol{X}^N}) \\
    & \leq 2C^2(s - \underline{s})+ 4L^2 | e_s^{(i)} |^2 + 4L^2 | \hat{X}_{s}^{i,N} - \bar{X}_{s}^{i,N} |^2 + 4L^2 \left(\mathcal{W}^{2}_2(\mu_{s}^{\bar{\boldsymbol{X}}^N},\mu_{s}^{\hat{\boldsymbol{X}}^N}) + \mathcal{W}^{2}_2(\mu_{s}^{\hat{\boldsymbol{X}}^N},\mu_{s}^{\boldsymbol{X}^N}) \right) \\
    & \leq 2C^2(s - \underline{s}) + 4L^2 | e_s^{(i)} |^2 + 4L^2 | \hat{X}_{s}^{i,N} - \bar{X}_{s}^{i,N} |^2 + 
     \frac{4L^2}{N} \Bigg( \sum_{j=1}^N | \bar{X}_{s}^{j,N} - \hat{X}_{s}^{j,N} |^2 +  \sum_{j=1}^N | \hat{X}_{s}^{j,N} - X_{s}^{j,N}  |^2 \Bigg).
\end{align*}

Hence, putting all these estimates from above together, we get for some constant $C_{L}>0$, depending only on the (different) Lipschitz constants,
\allowdisplaybreaks
\begin{align*}
    | e_t^{(i)} |^2 &\leq C_{L} \int_{0}^t |  e_s^{(i)} |^2 \, \mathrm{d}s + \int_{0}^t (4L^2 +\tilde{L}(\hat{X}_{s}^{i,N},\bar{X}_{s}^{i,N})^2) | \bar{X}_{s}^{i,N} -\hat{X}_{s}^{i,N} |^2 \, \mathrm{d}s \\
    & \quad + 2 \int_{0}^t \left \langle  e_s^{(i)}, \left( \sigma(\underline{s}, \bar{X}_{s}^{i,N}, \mu_{s}^{\bar{\boldsymbol{X}}^N}) -\sigma(s, X_{s}^{i,N}, \mu_{s}^{\boldsymbol{X}^N}) \right) \, \mathrm{d}W_s^{i} \right \rangle \\
    & \quad + C_{L} \frac{1}{N} \int_{0}^t \sum_{j=1}^N | \bar{X}_{s}^{j,N} - \hat{X}_{s}^{j,N} |^2 \, \mathrm{d}s 
    \; + \; C_{L} \frac{1}{N} \int_{0}^t \sum_{j=1}^N | e_s^{(j)} |^2 \, \mathrm{d}s + \left(2TC^2 + \frac{C^2}{2}T^2 \right) \delta,
\end{align*}
where we used $(s-\underline{s}) \leq T\delta$. Hence, using Jensen's inequality, we get, for some constant $C_{T,p,L}>0$, the estimate 
\allowdisplaybreaks
\begin{align*}
   | e_s^{(i)} |^p &\leq C_{T,p,L}  \int_{0}^t  | e_s^{(i)} |^p  \, \mathrm{d}s + C_{T,p,L} \left(\frac{1}{N}\right)^{p/2} \int_{0}^t \left(\sum_{j=1}^N | e_s^{(j)} |^2 \right)^{p/2} \, \mathrm{d}s  \\
    & \quad + C_{T,p,L} \int_{0}^t  (4L^2 +\tilde{L}(\hat{X}_{s}^{i,N},\bar{X}_{s}^{i,N})^2))^{p/2} | \bar{X}_{s}^{i,N} -\hat{X}_{s}^{i,N} |^p \, \mathrm{d}s \\
    & \quad + C_{T,p,L} \left| \int_{0}^t \left \langle e_s^{(i)}, \left( \sigma(\underline{s}, \bar{X}_{s}^{i,N}, \mu_{s}^{\bar{\boldsymbol{X}}^N}) -\sigma(s, X_{s}^{i,N}, \mu_{s}^{\boldsymbol{X}^N}) \right) \, \mathrm{d}W_s^{i} \right \rangle \right|^{p/2} \\
    & \quad + C_{T,p,L} \left(\frac{1}{N}\right)^{p/2} \int_{0}^t \left(\sum_{j=1}^N | \bar{X}_{s}^{j,N} - \hat{X}_{s}^{j,N} |^2 \right)^{p/2} \, \mathrm{d}s + C_{T,p,L} \delta^{p/2}.
\end{align*}
Note that 
\begin{align*}
\int_{0}^t \left(\sum_{j=1}^N | \bar{X}_{s}^{j,N} - \hat{X}_{s}^{j,N} |^2 \right)^{p/2} \, \mathrm{d}s \leq N^{p/2-1} \int_{0}^t \sum_{j=1}^N | \bar{X}_{s}^{j,N} - \hat{X}_{s}^{j,N} |^p  \, \mathrm{d}s,
\end{align*}
and similarly for the other expressions in the previous estimate. Consequently, taking the supremum and expectation on both sides, we arrive at
\allowdisplaybreaks
\begin{align}\label{strongerror}
      \mathbb{E} \left[  \sup_{s \in [0,t]} | e_s^{(i)} |^p \right] &\leq C_{T,p,L} \int_{0}^t \mathbb{E} \left [  \sup_{u \in [0,s]}  | e_u^{(i)} |^p  \right] \, \mathrm{d}s \nonumber + C_{T,p,L} \int_{0}^t \mathbb{E} \left[  (4L^2 +\tilde{L}(\hat{X}_{s}^{i,N},\bar{X}_{s}^{i,N})^2)^{p/2} | \bar{X}_{s}^{i,N} -\hat{X}_{s}^{i,N} |^p \right] \, \mathrm{d}s \nonumber \\
    & \quad + C_{T,p,L} \mathbb{E} \left[  \sup_{s \in [0,t]}  \left| \int_{0}^s \left \langle e_u^{(i)}, \left( \sigma(\underline{u}, \bar{X}_{u}^{i,N}, \mu_{u}^{\bar{\boldsymbol{X}}^N}) -\sigma(u, X_{u}^{i,N}, \mu_{u}^{\boldsymbol{X}^N}) \right) \, \mathrm{d}W_u^{i} \right \rangle \right|^{p/2} \right] \nonumber \\
    & \quad + C_{T,p,L} \int_{0}^t \mathbb{E} \left[  | \bar{X}_{s}^{i,N} - \hat{X}_{s}^{i,N} |^{p} \right] \, \mathrm{d}s   +  C_{T,p,L} \delta^{p/2}.
\end{align}

Applying the Burkholder-Davis-Gundy inequality in combination with Jensen's inequality and the global Lipschitz continuity of $\sigma$, we get for some constant $C_{T,p,L} > 0$. 
\allowdisplaybreaks
\begin{align*}
   & \mathbb{E} \left[   \sup_{s \in [0,t]} \left| \int_{0}^s \left \langle e_u^{(i)}, \left( \sigma(\underline{u}, \bar{X}_{u}^{i,N}, \mu_{u}^{\bar{\boldsymbol{X}}^N}) -\sigma(u, X_{u}^{i,N}, \mu_{u}^{\boldsymbol{X}^N}) \right) \, \mathrm{d}W_u^{i} \right \rangle \right|^{p/2} \right] \\
   & \leq C_{T,p,L} \mathbb{E} \left[  \int_{0}^{t} | e_s^{(i)} |^{p/2} \left( |\bar{X}_{s}^{i,N} -X_{s}^{i,N} | + \mathcal{W}_2(\mu_{s}^{\bar{\boldsymbol{X}}^N} , \mu_{s}^{\boldsymbol{X}^N}) + (T\delta)^{1/2} \right)^{p/2} \, \mathrm{d}s \right] \\
   & \leq C_{T,p,L} \mathbb{E} \left[ \int_{0}^{t}\frac{1}{2}  |  e_s^{(i)}  |^{p} + \frac{1}{2} \left( |\bar{X}_{s}^{i,N} -X_{s}^{i,N}   | + \mathcal{W}_2(\mu_{s}^{\bar{\boldsymbol{X}}^N} , \mu_{s}^{\boldsymbol{X}^N}) + (T\delta)^{1/2} \right)^{p} \, \mathrm{d}s \right] \\
   &\leq C_{T,p,L} \mathbb{E} \left[  \int_{0}^{t}  |  e_s^{(i)}  |^{p} +  | \bar{X}_{s}^{i,N} -\hat{X}_{s}^{i,N} |^p +  \mathcal{W}_2^{p}(\mu_{s}^{\bar{\boldsymbol{X}}^N} , \mu_{s}^{\boldsymbol{X}^N})  \, \mathrm{d}s  + (T \delta)^{p/2} \right] \\
   & \leq C_{T,p,L} \mathbb{E} \Big[  \int_{0}^{t}  \left(|  e_s^{(i)} |^{p} + | \hat{X}_{s}^{i,N} -\bar{X}_{s}^{i,N} |^p +  \mathcal{W}_2^{p}(\mu_{s}^{\bar{\boldsymbol{X}}^N} , \mu_{s}^{\hat{\boldsymbol{X}}^N}) +  \mathcal{W}_2^{p}(\mu_{s}^{\hat{\boldsymbol{X}}^N}, \mu_{s}^{\boldsymbol{X}^N}) \right) \, \mathrm{d}s + (T \delta)^{p/2} \Big].
\end{align*}
Further, recall that
\begin{align*}
    & \mathcal{W}_2^{p}(\mu_{s}^{\bar{\boldsymbol{X}}^N},\mu_{s}^{\hat{\boldsymbol{X}}^N}) \leq  \frac{1}{N} \sum_{j=1}^N | \bar{X}_{s}^{j,N} -\hat{X}_{s}^{j,N} |^p, \quad
    & \mathcal{W}_2^{p}(\mu_{s}^{\hat{\boldsymbol{X}}^N},\mu_{s}^{\boldsymbol{X}^N}) \leq  \frac{1}{N} \sum_{j=1}^N | e_s^{(j)} |^p.
\end{align*}
Hence, there exists a constant $C_{T,p,L} > 0$,
\begin{align*}
   & \mathbb{E} \left[   \sup_{s \in [0,t]} \left| \int_{0}^s \left \langle e_u^{(i)} , \left( \sigma(u, \bar{X}_{u}^{i,N}, \mu_{u}^{\bar{\boldsymbol{X}}^N}) -\sigma(u, X_{u}^{i,N}, \mu_{u}^{\boldsymbol{X}^N}) \right) \, \mathrm{d}W_u^{i} \right \rangle \right|^{p/2} \right] \\
   & \leq C_{T,p,L} \mathbb{E} \Bigg[ \int_{0}^{t}  |  e_s^{(i)}  |^{p} +  | \bar{X}_{s}^{i,N} -\hat{X}_{s}^{i,N} |^p  \, \mathrm{d}s  + (T \delta)^{p/2} \Bigg].
\end{align*}

The claim follows then from (\ref{strongerror}) using standard arguments, since for the error terms of the form $|  e_s^{(i)} |^{p}$ one can apply Gronwall's inequality and the expected values of the terms of the form $| \bar{X}_{s}^{i,N} -\hat{X}_{s}^{i,N} |^p$ are of order $\delta^{p/2}$; which can be shown using a conditional expectation argument for the Brownian increment and taking the growth of the coefficients along with Proposition \ref{prop1} into account.
\end{proof}

\subsection{Convergence of Scheme 2 (Proof of Theorem \ref{thm:main})}
\label{sec:main}

\begin{proof}
Let $X^{i,N}$ be given by \eqref{eq:PS},
for any $i \in \lbrace 1, \ldots, N \rbrace$.
It is shown in \cite[Proposition 3.1]{RES} that
$\max_{i \in \lbrace 1, \ldots, N \rbrace}  \mathbb{E}\left[ \sup_{t \in [0,T]} | X_t^{i} - {X}_t^{i,N} |^2 \right] \leq C \varphi(N)$
for some $C$ independent of $N$.
It therefore remains to show here that
 $\max_{i \in \lbrace 1, \ldots, N \rbrace}  \mathbb{E}\left[ \sup_{t \in [0,T]} | X_t^{i,N} - \tilde{X}_t^{i,N} |^2 \right] \leq C \delta$
for some $C$ independent of $N$ and $\delta$.
 
In what follows, we will drop for ease of notation the explicit time-dependence of the coefficients. Then, we may write for $t \in [t_n, t_{n+1}]$ for some $n \geq 1$,
\begin{align*}
|X_t^{i,N} - \tilde{X}_t^{i,N}|^2 &= \Big| X_{\underline{t}}^{i,N} - \tilde{X}_{\underline{t}}^{i,N} + \int_{\underline{t}}^{t} b(X_{s}^{i,N}, \mu_s^{\boldsymbol{X}^N}) \, \mathrm{d}s - b(\tilde{X}_{\underline{t}}^{i,N},\mu_{k_n\delta T}^{\tilde{\boldsymbol{X}}^N})(t-\underline{t}) \\
& \qquad \int_{\underline{t}}^{t} \sigma(X_{s}^{i,N}, \mu_s^{\boldsymbol{X}^N}) \, \mathrm{d}W^{i}_s - \sigma(\tilde{X}_{\underline{t}}^{i,N}, \mu_{k_n\delta T}^{\tilde{\boldsymbol{X}}^N})(W_t^{i}-W_{\underline{t}}^{i}) \Big|^2.
\end{align*}
Squaring the term on the right side and taking expectations yields
\begin{align}
\nonumber
 \mathbb{E}[|X_t^{i,N} - \tilde{X}_t^{i,N}|^2] &\leq  \mathbb{E}[| X_{\underline{t}}^{i,N} - \tilde{X}_{\underline{t}}^{i,N}|^2] + 2 \mathbb{E}\left[\left|\int_{\underline{t}}^{t} b(X_{s}^{i,N}, \mu_s^{\boldsymbol{X}^N}) \, \mathrm{d}s - b(\tilde{X}_{\underline{t}}^{i,N}, \mu_{k_n\delta T}^{\tilde{\boldsymbol{X}}^N})(t-\underline{t}) \right|^2 \right] \\
\nonumber
& \quad + 2\mathbb{E}\left[ \Big|\int_{\underline{t}}^{t} \sigma(X_{s}^{i,N}, \mu_s^{\boldsymbol{X}^N}) \, \mathrm{d}W^{i}_s - \sigma(\tilde{X}_{\underline{t}}^{i,N}, \mu_{k_n\delta T}^{\tilde{\boldsymbol{X}}^N})(W_t^{i}-W_{\underline{t}}^{i}) \Big|^2 \right] \\
& \quad + 2\mathbb{E} \left \langle X_{\underline{t}}^{i,N} - \tilde{X}_{\underline{t}}^{i,N} ,  \int_{\underline{t}}^{t} b(X_{s}^{i,N}, \mu_s^{\boldsymbol{X}^N}) \, \mathrm{d}s - b(\tilde{X}_{\underline{t}}^{i,N}, \mu_{k_n\delta T}^{\tilde{\boldsymbol{X}}^N})(t-\underline{t}) \right \rangle.
\label{last_term}
\end{align}
Now, observe that H\"{o}lder's inequality implies
\begin{align*}
& \hspace{-2 cm} \mathbb{E}\left[\left|\int_{\underline{t}}^{t} b(X_{s}^{i,N}, \mu_s^{\boldsymbol{X}^N}) \, \mathrm{d}s - b(\tilde{X}_{\underline{t}}^{i,N}, \mu_{k_n\delta T}^{\tilde{\boldsymbol{X}}^N})(t-\underline{t}) \right|^2 \right]  \\
&  \leq C \mathbb{E}\left[(t - \underline{t}) \int_{\underline{t}}^{t} |b(X_{s}^{i,N}, \mu_s^{\boldsymbol{X}^N})|^2 \, \mathrm{d}s \right] +C\mathbb{E}\left[ |b(\tilde{X}_{\underline{t}}^{i,N},\mu_{k_n\delta T}^{\tilde{\boldsymbol{X}}^N})|^2 |t-\underline{t}|^2  \right]  \leq C \delta^2,
\end{align*}
for some constant $C>0$, due to Assumptions (H.\ref{Assum:Ax}(\ref{Assum:Ax4})) and (H.\ref{Assum:Ax}(\ref{Assum:Ax6})) and the moment stability of $X^{i,N}$ and $\tilde{X}^{i,N}$. Also, recall that $t-\underline{t} \leq \delta T$.

Next, we bound
\begin{align*}
& \hspace{-1 cm} \mathbb{E}\left[\left|\int_{\underline{t}}^{t} \sigma(X_{s}^{i,N}, \mu_s^{\boldsymbol{X}^N}) \, \mathrm{d}W^{i}_s - \sigma(\tilde{X}_{\underline{t}}^{i,N}, \mu_{k_n\delta T}^{\tilde{\boldsymbol{X}}^N})(W_t^{i}-W_{\underline{t}}^{i})\right|^2 \right]
 \; \leq \; \sum_{i=1}^{3} \Pi_i, \qquad \text{where} \\
\Pi_1 &= \ 3 \mathbb{E}\left[\left|\int_{\underline{t}}^{t} \sigma(X_{s}^{i,N}, \mu_s^{\boldsymbol{X}^N}) \, \mathrm{d}W^{i}_s - \sigma(X_{\underline{t}}^{i,N},  \mu_{\underline{t}}^{\boldsymbol{X}^N})(W_t^{i}-W_{\underline{t}}^{i})  \right|^2 \right],  \\
\Pi_2 &= \ 3\mathbb{E}\left[\left|\sigma(X_{\underline{t}}^{i,N},  \mu_{\underline{t}}^{\boldsymbol{X}^N})(W_t^{i}-W_{\underline{t}}^{i}) -  \sigma(\tilde{X}_{\underline{t}}^{i,N}, \mu_{\underline{t}}^{\tilde{\boldsymbol{X}}^N})(W_t^{i}-W_{\underline{t}}^{i}) \right|^2 \right], \\
\Pi_3 &=  \ 3\mathbb{E}\left[\left| \sigma(\tilde{X}_{\underline{t}}^{i,N}, \mu_{\underline{t}}^{\tilde{\boldsymbol{X}}^N})(W_t^{i}-W_{\underline{t}}^{i})  - \sigma(\tilde{X}_{\underline{t}}^{i,N}, \mu_{k_n\delta T}^{\tilde{\boldsymbol{X}}^N})(W_t^{i}-W_{\underline{t}}^{i})  \right|^2 \right].
\end{align*}
By virtue of Assumption (H.\ref{Assum:Ax}(\ref{Assum:Ax1})), we get the estimate
\begin{align*}
& \| \sigma(X_{t}^{i,N}, \mu_t^{\boldsymbol{X}^N}) - \sigma(X_{\underline{t}}^{i,N},  \mu_{\underline{t}}^{\boldsymbol{X}^N}) \|^2 \leq C \left( |X_{t}^{i,N} - X_{\underline{t}}^{i,N}|^2 + \mathcal{W}_2^{2}( \mu_t^{N}, \mu_{\underline{t}}^{\boldsymbol{X}^N}) \right).
\end{align*}
Recalling the bound
\begin{equation*}
 \mathcal{W}_2^{2}( \mu_t^{\boldsymbol{X}^N}, \mu_{\underline{t}}^{\boldsymbol{X}^N}) \leq \frac{1}{N} \sum_{j=1}^{N} |X_{t}^{j,N} - X_{\underline{t}}^{j,N}|^2,
\end{equation*}
along with the standard one-step estimate $\mathbb{E}[\sup_{s \in [\underline{t},t]}|X_{s}^{i,N}-X_{\underline{t}}^{i,N}|^2] \leq C\delta$ and It\^{o}'s isometry, allows us to show $\Pi_1 \leq C \delta^2$, for some constant $C>0$. 
Using again (H.\ref{Assum:Ax}(\ref{Assum:Ax1})), a conditional expectation argument and estimates for Brownian increments, we derive 
\begin{align*}
\Pi_2 \leq C \delta \mathbb{E}[| X_{\underline{t}}^{i,N} - \tilde{X}_{\underline{t}}^{i,N}|^2].
\end{align*}
It remains to analyse 
$\Pi_3$,
where, for ease of notation, we consider $k_n=1$ and $T=1$ so that $\delta T = \delta$. Hence, we observe for $s\in [\delta,2\delta)$
\begin{align}\label{eq:MS}
   \mathbb{E}[| \tilde{X}_{s}^{i,N} -\tilde{X}_{\delta}^{i,N}  |^2]  
  &= \mathbb{E}[| \tilde{X}_{\underline{s}}^{i,N} + b(\underline{s},\tilde{X}^{i,N}_{\underline{s}}, \mu^{\tilde{\boldsymbol{X}}^N}_{\delta})(s-\underline{s}) +\sigma(\underline{s},\tilde{X}^{i,N}_{\underline{s}}, \mu^{\tilde{\boldsymbol{X}}^N}_{\delta})(W_s^i-W_{\underline{s}}^i) -\tilde{X}_{\delta}^{i,N}  |^2]  \notag \\
   & = \mathbb{E} \left[ \left | \int_{\delta}^s  b(\underline{u},\tilde{\bar{X}}^{i,N}_{u},  \mu^{\tilde{\bar{\boldsymbol{X}}}^N}_{u}) \, \mathrm{d}u + \int_{\delta}^s \sigma(\underline{u},\tilde{\bar{X}}^{i,N}_{u},  \mu^{\tilde{\bar{\boldsymbol{X}}}^N}_{u}) \, \mathrm{d}W_u^{i} \ \right |^2 \right]  \notag \\
   & \leq 2\mathbb{E} \left[ \left | \int_{\delta}^s  b(\underline{u},\tilde{\bar{X}}^{i,N}_{u},  \mu^{\tilde{\bar{\boldsymbol{X}}}^N}_{u}) \, \mathrm{d}u \right |^2 \right] + 2 \mathbb{E} \left[  \int_{\delta}^s  \| \sigma(\underline{u},\tilde{\bar{X}}^{i,N}_{u},  \mu^{\tilde{\bar{\boldsymbol{X}}}^N}_{u}) \|^2 \, \mathrm{d}u   \right]  \notag \\
   & \leq 2\mathbb{E} \left[ \delta \int_{\delta}^{s} | b(\underline{u},\tilde{\bar{X}}^{i,N}_{u},  \mu^{\tilde{\bar{\boldsymbol{X}}}^N}_{u})  |^2 \, \mathrm{d}u \right] +  2 \mathbb{E} \left[  \int_{\delta}^s  \| \sigma(\underline{u},\tilde{\bar{X}}^{i,N}_{u},  \mu^{\tilde{\bar{\boldsymbol{X}}}^N}_{u}) \|^2 \, \mathrm{d}u   \right]  \notag \\
   & \leq C \delta,
\end{align}
for some $C>0$ and where the last inequality is due to the polynomial growth bounds on $b$ and $\sigma$ and the moment stability of the process $\tilde{X}^{i,N}$. Hence, (H.\ref{Assum:Ax}(\ref{Assum:Ax1})) and estimates for Brownian increments allow us to deduce that $\Pi_3 \leq C \delta^2$.  

Returning to \eqref{last_term},
we write the last term there as
\begin{align*}
& \hspace{-1 cm} \mathbb{E} \left \langle X_{\underline{t}}^{i,N} - \tilde{X}_{\underline{t}}^{i,N},  \int_{\underline{t}}^{t} b(X_{s}^{i,N}, \mu_s^{\boldsymbol{X}^N}) \, \mathrm{d}s - b(\tilde{X}_{\underline{t}}^{i,N}, \mu_{k_n\delta T}^{\tilde{\boldsymbol{X}}^N})(t-\underline{t}) \right \rangle \ = \ \sum_{i=1}^{3} \Lambda_i, \qquad \text{where} \\
\Lambda_1 & = \mathbb{E} \left \langle X_{\underline{t}}^{i,N} - \tilde{X}_{\underline{t}}^{i,N} ,  \int_{\underline{t}}^{t} b(X_{s}^{i,N}, \mu_s^{\boldsymbol{X}^N}) \, \mathrm{d}s - b(X_{\underline{t}}^{i,N},  \mu_{\underline{t}}^{\boldsymbol{X}^N})(t-\underline{t})  \right \rangle \\
\Lambda_2 & = \mathbb{E} \left \langle X_{\underline{t}}^{i,N} - \tilde{X}_{\underline{t}}^{i,N} ,  b(X_{\underline{t}}^{i,N},  \mu_{\underline{t}}^{\boldsymbol{X}^N})(t-\underline{t}) -  b(\tilde{X}_{\underline{t}}^{i,N}, \mu_{\underline{t}}^{\tilde{\boldsymbol{X}}^N})(t-\underline{t})  \right \rangle \\
\Lambda_3 & = \mathbb{E} \left \langle X_{\underline{t}}^{i,N} - \tilde{X}_{\underline{t}}^{i,N} , b(\tilde{X}_{\underline{t}}^{i,N}, \mu_{\underline{t}}^{\tilde{\boldsymbol{X}}^N})(t-\underline{t})  - b(\tilde{X}_{\underline{t}}^{i,N}, \mu_{k_n\delta T}^{\tilde{\boldsymbol{X}}^N})(t-\underline{t})   \right \rangle.
\end{align*}
In the sequel, we will analyse these terms one-by-one. First, using Young's inequality and H\"{o}lder's inequality plus taking (H.\ref{Assum:Ax}(\ref{Assum:Ax3})) and (H.\ref{Assum:Ax}(\ref{Assum:Ax4})) into account, we get
\begin{align*}
\Lambda_1 & \leq \mathbb{E} \left[ |X_{\underline{t}}^{i,N} - \tilde{X}_{\underline{t}}^{i,N}| \delta^{1/2}  \delta^{-1/2} \left| \int_{\underline{t}}^{t} b(X_{s}^{i,N}, \mu_s^{\boldsymbol{X}^N}) \, \mathrm{d}s - b(X_{\underline{t}}^{i,N},  \mu_{\underline{t}}^{\boldsymbol{X}^N})(t-\underline{t})  \right| \right] \\
& \leq C\delta  \mathbb{E} \left[|X_{\underline{t}}^{i,N} - \tilde{X}_{\underline{t}}^{i,N}|^2 \right] + C\delta^{-1} \mathbb{E}\left[ \left| \int_{\underline{t}}^{t} b(X_{s}^{i,N}, \mu_s^{\boldsymbol{X}^N}) \, \mathrm{d}s - b(X_{\underline{t}}^{i,N},  \mu_{\underline{t}}^{\boldsymbol{X}^N})(t-\underline{t}) \right|^2  \right] \\
& \leq  C\delta  \mathbb{E} \left[|X_{\underline{t}}^{i,N} - \tilde{X}_{\underline{t}}^{i,N}|^2 \right] + C\delta^{-1} \mathbb{E}\left[(t-\underline{t})\int_{\underline{t}}^{t} |b(X_{s}^{i,N}, \mu_s^{\boldsymbol{X}^N})-b(X_{\underline{t}}^{i,N},  \mu_{\underline{t}}^{\boldsymbol{X}^N})|^2 \, \mathrm{d}s   \right] \\
& \leq  C\delta  \mathbb{E} \left[|X_{\underline{t}}^{i,N} - \tilde{X}_{\underline{t}}^{i,N}|^2 \right] + C \mathbb{E}\left[\int_{\underline{t}}^{t} \left( (1 + |X_{s}^{i,N}|^q +|X_{\underline{t}}^{i,N}|^q)^2 |X_{s}^{i,N}-X_{\underline{t}}^{i,N}|^2 + \mathcal{W}_2^{2}(\mu_s^{\boldsymbol{X}^N},\mu_{\underline{t}}^{\boldsymbol{X}^N}) \right) \, \mathrm{d}s   \right] \\
& \leq  C\delta  \mathbb{E} \left[|X_{\underline{t}}^{i,N} - \tilde{X}_{\underline{t}}^{i,N}|^2 \right] + C\delta^2,
\end{align*}
where we also used the one-step estimate $\mathbb{E}[\sup_{s \in [\underline{t},t]}|X_{s}^{i,N}-X_{\underline{t}}^{i,N}|^p] \leq C\delta^{p/2}$ and the stability of $X^{i,N}$. By similar arguments, in combination with (H.\ref{Assum:Ax}(\ref{Assum:Ax2})), we derive
\begin{equation*}
 \Lambda_2 \leq C\delta  \mathbb{E} \left[|X_{\underline{t}}^{i,N} - \tilde{X}_{\underline{t}}^{i,N}|^2 \right].
\end{equation*} 
In a similar manner, recalling estimate (\ref{eq:MS}), we obtain for the last term 
\begin{align*}
\Lambda_3 \leq C\delta  \mathbb{E} \left[|X_{\underline{t}}^{i,N} - \tilde{X}_{\underline{t}}^{i,N}|^2 \right] + C\delta^2.
\end{align*}  
Collecting the above estimates, we conclude
\begin{align*}
\mathbb{E}[|X_t^{i,N} - \tilde{X}_t^{i,N}|^2] \leq  C(1+\delta)\mathbb{E}[| X_{\underline{t}}^{i,N} - \tilde{X}_{\underline{t}}^{i,N}|^2] + C\delta^2.
\end{align*}
An inductive argument over all time-steps in combination with Proposition \ref{PropTS} gives the claim.
\end{proof}

\subsection{Moment stability for super-linear diffusion coefficient}\label{sec:supDiff}
\begin{prop}\label{prop1N}
Let $p>0$ and $X_0 \in L^p_0(\mathbb{R}^d)$. Let Assumptions (H.\ref{Assum:B}) and (H.\ref{Assum:Ax}(\ref{Assum:Ax3}))--(H.\ref{Assum:Ax}(\ref{Assum:Ax6})) hold, and suppose the time-step function satisfies (H.\ref{Assum:BS2}). Then, $T$ is almost surely attainable, i.e. 
$\mathbb{P}(\exists \ M(\omega) < \infty, \text{ s.t. } t_{M(\omega)} \geq T)=1$. Moreover, there exists a constant $C>0$ such that  
\begin{equation*}
    \max_{i \in \lbrace 1, \ldots, N \rbrace} \sup_{t \in [0,T]} \mathbb{E}\left[ |  \hat{X}_{t}^{i,N} |^p \right] 
     \leq C.
\end{equation*}
\end{prop}
\begin{proof}
We get for some $C_p>0$, using equation (\ref{AdaptiveMcKean})
\begin{align*}
 |\hat{X}_{t}^{i,N} - \hat{X}_{\underline{t}}^{i,N}  |^{p} &\leq  C_p  | b(\underline{t},\bar{X}_{t}^{i,N},\mu_{t}^{\bar{\boldsymbol{X}}^N})  |^{p} (t-\underline{t})^p  +  C_p  |\sigma(\underline{t},\bar{X}_{t}^{i,N}, \mu_{t}^{\bar{\boldsymbol{X}}^N}) (W_{t}^i -W_{\underline{t}}^i)  |^{p} \nonumber,
\end{align*}
which on multiplying with the drift and taking expectations on both sides yields 
\begin{align}\label{eq:onestep}
&\mathbb{E} \left[  |\hat{X}_{t}^{i,N} - \hat{X}_{\underline{t}}^{i,N}  |^{p} | b(\underline{t},\bar{X}_{t}^{i,N}, \mu_{t}^{\bar{\boldsymbol{X}}^N})  |^{p}  \right] \notag \\
& \leq C_p \mathbb{E} \left[  | b(\underline{t},\bar{X}_{t}^{i,N}, \mu_{t}^{\bar{\boldsymbol{X}}^N})  |^{2p} (t-\underline{t})^p \right] + C_p \mathbb{E} \left[  |\sigma(\underline{t},\bar{X}_{t}^{i,N}, \mu_{t}^{\bar{\boldsymbol{X}}^N}) (W_{t}^i -W_{\underline{t}}^i) |^{p} | b(\underline{t},\bar{X}_{t}^{i,N}, \mu_{t}^{\bar{\boldsymbol{X}}^N})  |^{p}  \right] \notag \\
& \leq C_p \mathbb{E} \left[  (1+|\hat{X}_{\underline{t}}^{i,N}|^2)^{p/2} \right] + C_p \mathbb{E} \left[ \mathbb{E} \left[ |\sigma(\underline{t},\bar{X}_{t}^{i,N}, \mu_{t}^{\bar{\boldsymbol{X}}^N}) (W_{t}^i -W_{\underline{t}}^i)    |^{p} | b(\underline{t},\bar{X}_{t}^{i,N}, \mu_{t}^{\bar{\boldsymbol{X}}^N})  |^{p}  \big|  \mathcal{F}_{\underline{t}}  \right] \right] \notag \\
& \leq C_p \mathbb{E} \left[  (1+|\hat{X}_{\underline{t}}^{i,N}|^2)^{p/2} \right],
\end{align}
where we used $\mathbb{E} \left[ |W_{t}^i -W_{\underline{t}}^i |^{p} \big| \mathcal{F}_{\underline{t}} \right] \leq C_p (t-\underline{t})^{p/2}$, for some $C_p > 0$, along with (H.\ref{Assum:BS2}) and Remark \ref{remark:SupeDiff}. This one-step error estimate will be essential at a later stage of the proof.

Now, we consider the stopping time $\tau_{R} := \min_{i \in \lbrace 1, \ldots, N\rbrace} \tau^{i}_{R}$, where $\tau^{i}_{R}: = \inf \lbrace t \geq 0 \ : \  |\hat{X}_{t}^{i,N} | \geq R \rbrace$. Due to (H.\ref{Assum:BS2}), $T \land \tau_{R}$ is attainable. Therefore, It\^{o}'s formula for the stopped process $\hat{X}_{t \land \tau_{R}}^{i,N}$ gives
\begin{align}
 \big(1 + |\hat{X}_{t \land \tau_{R}}^{i,N} |^{2} \big)^{p/2} & =  \left( 1 + |\hat{X}_{0}^{i,N} |^{2} \right)^{p/2}  
+ p   \int_{0}^{t \land \tau_{R}} \left(1+ |\hat{X}_{s}^{i,N} |^2 \right)^{p/2 -1} \left \langle \hat{X}_{s}^{i,N}, b(\underline{s},\bar{X}_{s}^{i,N}, \mu_{s}^{\bar{\boldsymbol{X}}^N}) \right \rangle \, \mathrm{d}s \nonumber 
\\
&  \quad + p   \int_{0}^{t \land \tau_{R}} \left(1+ |\hat{X}_{s}^{i,N} |^2 \right)^{p/2 -1}  \left \langle \hat{X}_{s}^{i,N}, \sigma(\underline{s},\bar{X}_{s}^{i,N}, \mu_{s}^{\bar{\boldsymbol{X}}^N}) \, \mathrm{d}W_s^i \right \rangle \nonumber 
\\
&  \quad +  \frac{p(p-2)}{2}  \int_{0}^{t \land \tau_{R}} \left(1+ |\hat{X}_{s}^{i,N} |^2 \right)^{p/2 -2} \big |  \sigma^{\top}(\underline{s},\bar{X}_{s}^{i,N}, \mu_{s}^{\bar{\boldsymbol{X}}^N}) \hat{X}_{s}^{i,N}  \big |^2 \, \mathrm{d}s \nonumber 
\\ 
& \quad +  \frac{p}{2}  \int_{0}^{t \land \tau_{R}} \left(1+ |\hat{X}_{s}^{i,N} |^2 \right)^{p/2 -1} \big \| \sigma(\underline{s},\bar{X}_{s}^{i,N}, \mu_{s}^{\bar{\boldsymbol{X}}^N})  \big \|^2 \, \mathrm{d}s, \nonumber 
\end{align}
almost surely for any $t\in[0,T]$. Thus, on taking expectation and using the Cauchy-Schwarz inequality 
\begin{align}
& \mathbb{E} \left[ \big(1 + | \hat{X}_{t \land \tau_{R}}^{i,N}|^{2} \big)^{p/2} \right] 
\leq \mathbb{E} \left[  \big( 1 + | \hat{X}_{0}^{i,N} |^{2} \big)^{p/2} \right] \notag \\
& \qquad\quad + p   \mathbb{E} \left[  \int_{0}^{t \land \tau_{R}}  \left(1+ |\hat{X}_{s}^{i,N} |^2 \right)^{p/2 -1} \left( \left \langle \bar{X}_{s}^{i,N},  b(\underline{s},\bar{X}_{s}^{i,N}, \mu_{s}^{\bar{\boldsymbol{X}}^N}) \right \rangle  +  \frac{(p-1)}{2}   \big \| \sigma(\underline{s},\bar{X}_{s}^{i,N}, \mu_{s}^{\bar{\boldsymbol{X}}^N}) \big \|^2  \right) \, \mathrm{d}s  \right] \notag
\\  
& \qquad\quad + p  \mathbb{E} \left[ \int_{0}^{t \land \tau_{R}}  \left(1+ | \hat{X}_{s}^{i,N} |^2 \right)^{p/2 -1} \left \langle  \hat{X}_{s}^{i,N} - \hat{X}_{\underline{s}}^{i,N}, b(\underline{s},\bar{X}_{s}^{i,N}, \mu_{s}^{\bar{\boldsymbol{X}}^N}) \right \rangle  \, \mathrm{d}s \right]. \notag
\end{align}
Hence, by (H.\ref{Assum:B}(\ref{Assum:B2})) we further obtain, for some constant $C_{p,L} >0$ (depending on the constants in (H.\ref{Assum:B}(\ref{Assum:B2})))    
\begin{align}
\mathbb{E} \left[\big(1 + |X_{t \land \tau_{R}}^{i,N, n} |^{2} \big)^{p_0/2} \right]  & \leq \mathbb{E} \left[ \big( 1 + |\hat{X}_{0}^{i,N}|^{2} \big)^{p/2} \right] + C_{p,L} \int_{0}^{t} \sup_{r \in [0,s]} \mathbb{E} \left[ \big(1+ |\hat{X}_{r \land \tau_{R}}^{i,N}|^2 \big)^{p/2} \right]  \, \mathrm{d}s + pF, \label{eq:F1+F2}
\end{align}
for any $t\in[0,T]$, where $F$ is is given by  
\begin{align*}
F & := \mathbb{E} \left[ \int_{0}^{t \land \tau_{R}}  \left(1+ |\hat{X}_{s}^{i,N}|^2 \right)^{p/2 -1}\left \langle \hat{X}_{s}^{i,N} - \hat{X}_{\underline{s}}^{i,N}, b(\underline{s},\bar{X}_{s}^{i,N}, \mu_{s}^{\bar{\boldsymbol{X}}^N}) \right \rangle   \, \mathrm{d}s \right]. \nonumber
\end{align*}
This gives, using H\"{o}lder's inequality along with equation (\ref{eq:onestep}) and (H.\ref{Assum:BS2}), 
\begin{align}\label{eq:F1:ms}
F  & \leq  C_p \int_{0}^{t} \sup_{r \in [0,s]} \mathbb{E} \left[ \big(1+ |\hat{X}_{r \land \tau_{R}}^{i,N}|^2 \big)^{p/2} \right]   + C_ p\mathbb{E} \left[ \int_{0}^{t \land \tau_{R}}  | \hat{X}_{s}^{i,N} - \hat{X}_{\underline{s}}^{i,N} |^{p/2} | b(\underline{s},\bar{X}_{s}^{i,N}, \mu_{s}^{\bar{\boldsymbol{X}}^N}) |^{p/2} \, \mathrm{d}s \right] \notag \\
& \leq  C_p \int_{0}^{t} \sup_{r \in [0,s]} \mathbb{E} \left[ \big(1+ |\hat{X}_{r \land \tau_{R}}^{i,N}|^2 \big)^{p/2} \right].
\end{align}
for some generic constant $C_p>0$.
Combining equations (\ref{eq:F1+F2}) and (\ref{eq:F1:ms}) and using Gronwall's inequality yields the claim, up to time $T \land \tau_{R}$. 
Now, note that 
\begin{equation*}
\mathbb{P}(\tau_{R} \leq T) \leq \frac{1}{R^2} \mathbb{E}\left[\max_{i \in \lbrace 1, \ldots, N \rbrace} |\hat{X}_{T \land \tau_{R}}^{i,N}|^2 \right] \leq \frac{CN}{R^2}, 
\end{equation*}
and hence,
\begin{align*}
\mathbb{P} \left(  \max_{i \in \lbrace 1, \ldots, N \rbrace}  \sup_{t \in [0,T]} |\hat{X}_{t}^{i,N}| < R \right) = 1 - \mathbb{P}(\tau_{R} \leq T) \to 1, \text{ as } R \to \infty.  
\end{align*}
Therefore, almost surely, $ \max_{i \in \lbrace 1, \ldots, N \rbrace} \sup_{t \in [0,T]} |\hat{X}_{t}^{i,N}| < \infty$ and $T$ is attainable. Using this and the stability up to time $T \land \tau_{R}$, the claim follows from Fatou's lemma.
\end{proof}

\subsection{Moment stability of Milstein scheme} \label{proofMil}

\begin{prop}\label{prop_milstab}
Let $p>0$ and $\xi \in L^p_0(\mathbb{R}^d)$. Let Assumption (H.\ref{Assum:AxM}) hold, and suppose the time-step function satisfies (H.\ref{Assum:AxS}(\ref{Assum:AxS2})). Then, $T$ is almost surely attainable, i.e. 
$\mathbb{P}(\exists \ M(\omega) < \infty, \text{ s.t. } t_{M(\omega)} \geq T)=1$.
Moreover, there exists a constant $C>0$ such that for $\hat{Y}$ defined in \eqref{eq:MilAdap}
\begin{equation*}
    \max_{i \in \lbrace 1, \ldots, N \rbrace} \mathbb{E}\left[\sup_{t \in [0,T]} |  \hat{Y}_{t}^{i,N} |^p \right]  \leq C.
\end{equation*}
\end{prop}

\begin{proof}

We outline the difference to the proof of the moment stability of the adaptive Euler-Maruyama scheme. The $K$-scheme for the Milstein discretisation reads as follows:
\begin{align*}
    \hat{Y}_{t_{n+1}}^{K,i,N} &= P_K \Big(\hat{Y}_{t_{n}}^{K,i,N} + \frac{1}{N} \sum_{j=1}^N b(\hat{Y}_{t_n}^{K,i,N}, \hat{Y}_{t_n}^{K,j,N})h_n^{\min} \\
    & \qquad + \sigma(\hat{Y}_{t_{n}}^{K,i,N}) \Delta W_{t_n}^{i} + \frac{1}{2} \sum_{j_1,j_2=1}^k L^{j_1} \sigma_{j_2}(\hat{Y}_{t_{n}}^{K,i,N}) (\Delta W_{t_n}^{i,j_1} \Delta W_{t_n}^{i,j_2}  -\delta_{j_1,j_2} h_n^{\min}) \Big).
\end{align*}
Hence, using
\begin{align*}
M_{t_n}^{K,i}:= \frac{1}{2} \sum_{j_1,j_2=1}^k L^{j_1} \sigma_{j_2}(\hat{Y}_{t_{n}}^{K,i,N}) (\Delta W_{t_n}^{i,j_1} \Delta W_{t_n}^{i,j_2}  -\delta_{j_1,j_2} h_n^{\min}),
\end{align*}
we get
\begin{align*}
    | \hat{Y}^{K,i,N}_{t_{n+1}} |^2 &\leq | \hat{Y}^{K,i,N}_{t_{n}}  |^2 + 2 h_n^{\min} \left( \left \langle \hat{Y}^{K,i,N}_{t_{n}} , \frac{1}{N} \sum_{j=1}^N b(\hat{Y}^{K,i,N}_{t_{n}},\hat{Y}^{j,N}_{t_{n}}) \right \rangle + \frac{1}{2}h_n^{\min} \left | \frac{1}{N} \sum_{j=1}^N b(\hat{Y}^{K,i,N}_{t_{n}},\hat{Y}^{j,N}_{t_{n}}) \right |^2  \right) \\
    & \quad + 2 \left \langle \phi(\hat{Y}_{t_n}^{K,i,N}), \sigma(\hat{Y}^{K,i,N}_{t_{n}}) \Delta W_{t_n}^i \right \rangle + | \sigma(\hat{Y}^{K,i,N}_{t_{n}}) \Delta W_{t_n}^i |^2 + | M_{t_n}^{K,i} |^2 \\
    & \quad + 2 \left \langle \hat{Y}^{K,i,N}_{t_{n}}, M_{t_n}^{K,i} \right \rangle + 2 \left \langle \sigma(\hat{Y}^{K,i,N}_{t_{n}}) \Delta W_{t_n}^i, M_{t_n}^{K,i}  \right \rangle + 2 h_n^{\min} \left \langle \frac{1}{N} \sum_{j=1}^N b(\hat{Y}^{K,i,N}_{t_{n}},\hat{Y}^{j,N}_{t_{n}}),  M_{t_n}^{K,i} \right \rangle \\
    & \leq  | \hat{Y}^{K,i,N}_{t_{n}} |^2 + 2 h_n^{\min} \left( \left \langle \hat{Y}^{K,i,N}_{t_{n}} , \frac{1}{N} \sum_{j=1}^N b(\hat{Y}^{K,i,N}_{t_{n}},\hat{Y}^{j,N}_{t_{n}}) \right \rangle + \frac{3}{2}h_n^{\min} \left | \frac{1}{N} \sum_{j=1}^N b(\hat{Y}^{K,i,N}_{t_{n}},\hat{Y}^{j,N}_{t_{n}}) \right|^2  \right) \\
    & \quad + 2 \left \langle \phi(\hat{Y}_{t_n}^{K,i,N}), \sigma(\hat{Y}^{K,i,N}_{t_{n}}) \Delta W_{t_n}^i \right \rangle + 2 | \sigma(\hat{Y}^{K,i,N}_{t_{n}}) \Delta W_{t_n}^i |^2 + 3 | M_{t_n}^{K,i} |^2 + 2 \left \langle \hat{Y}^{K,i,N}_{t_{n}}, M_{t_n}^{K,i} \right \rangle.
\end{align*}
Now, using the same techniques as in Proposition \ref{prop1}, the bound (see \cite{GW} for details)
\begin{align*}
\mathbb{E} \left[ \left( \sum_{n=0}^{n_t-1} |  M_{t_n}^{K,i} |^2 \right)^{p/2} \right] \leq C_{p,T,L} \int_{0}^{t} \mathbb{E} \left[ \sup_{u \in [0,s]} | X_u^{K,i,N}|^p \right] \, \mathrm{d}s + C_{p,T,L},
\end{align*}
for some constant $C_{p,T,L} > 0$, and Gronwall's inequality readily allows to prove stability of the truncated particle system. The attainability of $T$ can be argued as in the proof of Proposition \ref{prop1}. Then, a monotone convergence argument yields the claim for the particles $\hat{Y}^{i,N}$. 
\end{proof}

\end{document}